\documentclass[10pt]{amsart}
\usepackage{calrsfs}
\usepackage{soul}
\usepackage[greek,english]{babel}
\usepackage[iso-8859-7]{inputenc}
\usepackage{graphicx}
\usepackage{hyperref}
\usepackage{amsthm}
\usepackage{amssymb}
\usepackage{multirow}
\usepackage{tikz-cd}
\usepackage{amsmath}
\usepackage{todonotes}
\usepackage{amsbsy}
\usepackage[all]{xy}
\usepackage{changes}
\usepackage{enumerate}
\usepackage{comment}
\usepackage{mathtools}
\usepackage{changes}
\usepackage{mdframed}
\usepackage{lscape}
\definechangesauthor[color=orange,name={Aristides Kontogeorgis}]{AK}
 \definechangesauthor[color=red, name=Kostas Karagiannis]{KK}

\usepackage[margin=1 in]{geometry}
\usetikzlibrary{patterns,arrows,decorations.pathreplacing}

\newtheorem{theorem}{Theorem}
\newtheorem{lemma}[theorem]{Lemma}
\newtheorem{corollary}[theorem]{Corollary}
\newtheorem{proposition}[theorem]{Proposition}
\theoremstyle{definition}
\newtheorem{example}[theorem]{Example}

\newtheorem{definition}[theorem]{Definition}

\newcommand{\lc}{\left\lceil}
\newcommand{\rc}{\right\rceil}
\newcommand{\lf}{\left\lfloor}
\newcommand{\rf}{\right\rfloor}

\newcommand{\Z}{\mathbb{Z}}



\date{\today}

\title{The equivariant Hilbert series of the canonical ring of Fermat curves}

\author[H. Charalambous ]{Hara Charalambous  }
\address{
Department of Mathematics	,			
Aristotle University of Thessaloniki School of Sciences,
54124, Thessaloniki, Greece}
\email{hara@math.auth.gr}

\author[K. Karagiannis]{Kostas Karagiannis }
\address{Department of Mathematics, National and Kapodistrian  University of Athens
Pane\-pist\-imioupolis, 15784 Athens, Greece}
\email{konstantinos.v.karagiannis@gmail.com}

\author[S. Karanikolopoulos]{Sotiris Karanikolopoulos}
\address{Department of Mathematics, National and Kapodistrian  University of Athens
Pane\-pist\-imioupolis, 15784 Athens, Greece}
\email{sotiriskaran@gmail.com}

\author[A. Kontogeorgis]{Aristides Kontogeorgis}
\address{Department of Mathematics, National and Kapodistrian  University of Athens
Pane\-pist\-imioupolis, 15784 Athens, Greece}
\email{kontogar@math.uoa.gr}

\date \today

\makeatletter
\newcommand{\aprod}{\mathop{\operator@font \hbox{\Large$\ast$}}}
\makeatother

\begin{document}

\begin{abstract}
We consider a Fermat curve $F_n:x^n+y^n+z^n=1$ over an algebraically closed field $k$ of characteristic $p\geq0$ and study the action of the automorphism group $G=\left(\Z/n\Z\times\Z/n\Z\right)\rtimes S_3$ on the canonical ring $R=\bigoplus H^0(F_n,\Omega_{F_n}^{\otimes m})$ when $p>3$, $p\nmid n$ and $n-1$ is not a power of $p$. In particular, we explicitly determine the classes  $[H^0(F_n,\Omega_{F_n}^{\otimes m})]$ in the Grothendieck group $K_0(G,k)$ of finitely generated $k[G]$-modules, describe the respective equivariant Hilbert series $H_{R,G}(t)$ as a rational function, and use our results to write a program in Sage that computes $H_{R,G}(t)$ for an arbitrary Fermat curve.
\end{abstract}

\maketitle

\tableofcontents

\section{Introduction}
\subsection{Graded Representations}
Let $k$ be an algebraically closed field of characteristic $p\geq 0$, let $G$ be a finite group and let $\left(K_0(G,k),\oplus\right)$ denote the Grothendieck group of the category of finitely generated $k[G]$-modules; it is well-known, see for example \cite[Part III]{SerreLinear} that $K_0(G,k)$ is generated by the irreducible representations of $G$ over $k$ and that it becomes a commutative ring with unit with respect to $\otimes_K$. If $V$ is a $k[G]$-module, we denote by $[V]$ its image in $K_0(G,k)$ and recall that if ${\rm char}(k)\nmid |G|$, then $[V]$ determines uniquely the isomorphism class of $V$, whereas if ${\rm char}(k)\mid |G|$ this is no longer true.

Next, we consider a finitely generated, $\mathbb{N}$-graded $k$-algebra $R=\bigoplus_{d=0}^\infty R_d$ whose graded components $R_d$ are finite dimensional $k$-vector spaces acted on by $G$. The respective representations $\rho_d:G \rightarrow \mathrm{GL}(R_d)$ give rise to a formal power series
\[
H_{R,G}(t)=\sum_{d=0}^{\infty} [R_d] t^d,\text{ where }[R_d]\in K_0(G,k),
\]
which is called the {\em equivariant Hilbert series} of the action of $G$ on $R$. Note that if $G$ is trivial, then $[R_d]$ is just the dimension of $R_d$ as a $k$-vector space, and thus $H_{R,G}(t)$ generalizes the classic, non-equivariant Hilbert series of $R$. At the same time, it encodes all the invariants of the action of $G$ on $R$, which are infinitely many, in a finite, rational expression, see \cite{MR2525558} and \cite{MR2304322}. To explicitly compute $H_{R,G}(t)$ when ${\rm char}(k)\nmid G$, one can follow the approach described in Stanley's exposition \cite{MR526968}.

Each $[R_d]\in K_0(G,k)$ is uniquely determined by its decomposition $R_d =\bigoplus n_{d,V}V$ as a direct sum of irreducible $k[G]$-modules; this in turn gives rise to the following decomposition of the graded $k$-algebra $R$:
\[
R
=\bigoplus_{d=0}^\infty\bigoplus_{V\in{\rm Irrep(G)}}n_{d,V}V
=\bigoplus_{V\in{\rm Irrep(G)}}\bigoplus_{d=0}^\infty n_{d,V}V.
\]
The graded $k$-algebras $R_V^G:=\bigoplus_{d=0}^\infty n_{d,V}V$ are called the {\em isotypical components} of the action of $G$ on $R$ and we obtain the identity
\[
H_{R,G}(t)=\sum_{V\in{\rm Irrep}(G)}\sum_{d=0}^{\infty}n_{d,V}[V]t^d.
\]
Thus, the computation of the equivariant Hilbert series $H_{R,G}(t)$ is reduced to determining the multiplicities $n_{d,V}\in\mathbb{N}$ and studying the convergence of the respective power series for each $V\in{\rm Irrep}(G)$. This is theoretically doable using character theory, at least in the case of ordinary representations. However, computations become increasingly hard and thus one needs to take into account specific properties of $R$ and $G$ to get concrete results.\\

The most well studied case is when $R$ is the polynomial ring in $n$ variables over some algebraically closed field $k$ of characteristic $0$. The isotypical component corresponding to the trivial representation is then by definition the ring of invariants $R^G$ and every isotypical component $R_V^G$ for $V\in{\rm Irrep}(G)$ becomes naturally an $R^G$-module. Thus, the study of the equivariant Hilbert series in this context falls under classic invariant theory, while a beautiful result of Molien provides an explicit formula for $H_{R,G}$; for an overview of the subject and some striking applications to combinatorics the reader may refer to Stanley's exposition \cite{MR526968}.

Himstedt and Symonds in \cite{MR2525558} studied equivariant Hilbert series in a generalized setting by dropping the assumption on ${\rm char}(k)$ and considering finitely generated graded $R$-modules $M$ where $R$ in turn is a finitely generated $k$-algebra. This generalized setting gives a geometric flavor to graded representation theory, as the results of Himstedt and Symonds are applicable to line bundles $\mathcal{L}$ on projective curves which are equivariant under finite group actions. From this viewpoint, the study of equivariant Hilbert series relates to that of equivariant Euler characteristics, as expected from the non-equivariant case; the latter is an active area of research and has far-reaching applications, from normal integral bases of sheaf cohomology \cite{MR1274097}, to Dedekind zeta functions \cite{MR608528} to ramification theory \cite{Nak:86}, \cite{Koeck:04}, to mention a few.

This approach naturally gives a connection between equivariant Hilbert series and the classic problem of determining the Galois module structure of polydifferentials on projective curves. The problem was originally posed by Hecke and settled by Chevalley-Weil in \cite{Chevalley1934-eb} for characteristic $0$ curves; their results were generalized by Ellinsgrud and Lonsted in \cite{MR589254} when ${\rm char}(k)=p\nmid |G|$ while for modular representation theory the general case remains open and there exist only partial results. Finally, the case of integral representations, which naturally contains both ordinary and modular representation theory, has been studied only in very specific cases, see for example \cite{KaranProc} and \cite{MR1753111}.\\

To make things more concrete, we consider a pair $(X,G)$, where $X$ is a smooth, projective curve of genus $g\geq 4$ over $k$ which is not hyperelliptic, and $G$ is a finite subgroup of its automorphism group. If $\Omega_{X/k}$ denotes the sheaf of holomorphic differentials on $X$, then a classical result of Max Noether ensures that the {\em canonical map}
\[
\phi:{\rm Sym}\left(H^0(X,\Omega_{X/k}) \right)\twoheadrightarrow \bigoplus H^0(X,\Omega_{X/k}^{\otimes d})
\]
is surjective, giving rise to the {\em canonical embedding} $X\hookrightarrow\mathbb{P}_k^{g-1}$; for a modern treatment of the subject see \cite{Saint-Donat73}. The respective homogeneous coordinate ring $R=\bigoplus H^0(X,\Omega_{X/k}^{\otimes d})$, called {\em the canonical ring}, is a graded $k$-algebra acted on by $G$ and thus we can apply the results of the previous section. In this setting, the equivariant Hilbert series of interest is
\[
H_{R,G}(t)=\sum_{d=0}^{\infty} [H^0(X,\Omega_{X/k}^{\otimes d})] t^d
\]
and its computation requires determining the $k[G]$-module structure of the $k$-vector spaces $H^0(X,\Omega_{X/k}^{\otimes d})$ of holomorphic polydifferentials. To simplify the computations, we assume that ${\rm char}(k)=p\nmid |G|$; even though in theory there exist explicit formulas by the work of Chevalley-Weil and Ellinsgrud-Lonsted, there are many cases in which they cannot be used in practice, as they require explicit information on the ramification data of the cover $X\rightarrow X/G$. This paper treats an important class of cases falling under this category, namely the case when $X=F_n$ is a {\em Fermat curve} given by the equation $x^n+y^n+z^n=1$ and $G$ is its automorphism group.

If $n-1$ is not a power of the characteristic of the ground field $k$, then the automorphism group $G$ is isomorphic to $(\Z/n\Z \times \Z/n\Z) \rtimes S_3$, see \cite{Tze:95}, \cite{Leopoldt:96}. The case ${\rm char}(k)\mid n$ needs to be excluded as well, since in this case the Fermat curve is not reduced. We also exclude the characteristics $2,3$ so the representation we consider is ordinary. We can`t resist  to point out that the group $G$ appears as the analogue of $\mathrm{GL}_3(\mathbb{F}_{1^n})$, that is the general linear group with entries in the degree $n$ extension $\mathbb{F}_{1^n}$ of the mythical ``field'' with one element, see \cite{Kapranov1995-xx}. These groups play a significant role in knot theory \cite{GJKL}, and are isomorphic to the complex reflection groups $G(d,1,n)$, see \cite{MR2590895}.

Our main result is the explicit computation of the Galois module structure of holomorphic polydifferentials on Fermat curves, the calculation of the respective equivariant Hilbert series $H_{R,G}(t)$ and a computer program in Sage \cite{sage8.9} that computes $H_{R,G}(t)$ for an arbitrary Fermat curve. We conclude the introduction by giving an outline of our arguments and techniques.

\subsection{Outline} In Section \ref{sec:Irreps} we review some preliminaries on Fermat curves $F_n$ and their automorphism groups $G$. Since the latter are given as semidirect products, we use a classic result of Serre in Proposition \ref{prop:irr-reps-semidirect}, to construct all the irreducible representations of $G$. The list and the respective character table are given in Proposition \ref{prop:Irreps}. We proceed in Section \ref{sec:Chars} to obtain the characters of the action of $G$ on the global sections of holomorphic $m$-differentials. The standard bases which are given in the bibliography, see Proposition \ref{prop:bases-diffs}, are not suitable for computations; motivated by our work in \cite{charalambous2019relative} on the canonical embedding of smooth, projective curves, we rewrite $m$-differentials as a quotient of two $K$-vector spaces which are easier to manipulate and give the respective characters in Proposition \ref{prop:Wm-Im-chars}.

The inner products necessary to obtain the decomposition of $m$-differentials as a direct sum of the irreducible representations of Proposition \ref{prop:Irreps} reduce to computing various sums of roots unity which are interesting in their own sake. We thus devote Section \ref{sec:IJ} to their explicit computation and in Proposition \ref{prop:general-IJ} we obtain an equivalent characterization of these sums as the number lattice points inside a triangle which satisfy certain modular congruences. We then proceed with counting the cardinality of these lattices in Proposition \ref{prop:count-IJ} and obtain the exact values necessary for the subsequent sections in Corollary \ref{cor:count-IJ}.

Section \ref{sec:Polydiff} contains the explicit formulas for the Galois module structure of the global sections of holomorphic differentials. We compute the inner products of the irreducible characters of Section \ref{sec:Irreps} with the characters of Section \ref{sec:Chars} using the results of Section \ref{sec:IJ}. The main formulas are summarized in Theorem \ref{th:decompose-all}, and we verify our results in Table \ref{tab:dimension-comps} by giving the explicit decomposition of $m$-differentials for $m\in\{0,\ldots,9\}$ in the case of Fermat curves $F_4,F_5$ and $F_6$. Finally, in Section \ref{sec:GenFunctions} we obtain an explicit expression for the equivariant Hilbert series as a rational function: the main results are summarized in Theorem \ref{th:Hilbert} and we once more verify the computations by applying the results to get the equivariant Hilbert function in the case of the Fermat curve $F_6$. We remark that even though the formulas of Theorem \ref{th:decompose-all} and Theorem \ref{th:Hilbert} are complicated, they are appropriate for computations as they have allowed us to write a program in Sage which takes as input the value of $n$ that determines the Fermat curve $F_n$ and outputs the equivariant Hilbert function of its canonical ring. The code is uploaded in one of the authors' website\footnote{File: FinalCodeFermatReps.ipynb, url: \url{shorturl.at/fnzIR}}.

\section*{Acknowledgments}
\noindent Received financial support by program: ``Supporting researchers with emphasis to young researchers, cycle B'',  MIS 5047968.

\section{The irreducible representations of the automorphism group}\label{sec:Irreps}

Let $F_n: x_1^n+x_2^n+x_0^n=0$ be a Fermat curve defined over an algebraically closed field $K$ of characteristic $p\geq 0$. Setting $x=\frac{x_1}{x_0}$ and $y=\frac{x_2}{x_0}$ we obtain the affine model $F_n: x^n+y^n+1=0$. Let $g=(n-1)(n-2)/2$ 
be the genus of the Fermat curve. 
We are interested for Fermat curves of genus $g\geq 2$ that is $n\geq 4$. 
Further, we assume that $n\neq 1+p^h$ for all $h\in\mathbb{N}$, so that by 
\cite{Tze:95}, \cite{Leopoldt:96}
 the group of automorphisms of $F_n$ is given by
\[
G={\rm Aut}_K(F_n)=\left(\mathbb{Z}/n\mathbb{Z}\times\mathbb{Z}/n\mathbb{Z}\right)\rtimes S_3.
\]
To describe the action of $G$ on $F_n$, we write
\begin{eqnarray*}
A:=\mathbb{Z}/n\mathbb{Z}\times\mathbb{Z}/n\mathbb{Z}
&=&
\{
\sigma_{\alpha,\beta}:0\leq \alpha,\beta\leq n-1
\}\\
S_3&=&\langle s,t: s^3=t^2=1,\;tst=s^{-1}\rangle=\{1,s,t,st,ts,s^2 \}.
\end{eqnarray*}
Let $\zeta$ be an $n$-th root of unity. For $(x,y)\in F_n$, we have that $\sigma_{\alpha,\beta}(x,y)=\left(\zeta^\alpha x,\zeta^\beta y\right)$ and
\begin{eqnarray*}
s(x,y)&=&\left(\frac{y}{x},\frac{1}{x}\right),\;
t(x,y)=\left(\frac{1}{x},\frac{y}{x}\right),\;
st(x,y)=\left(\frac{x}{y},\frac{1}{y}\right),\;
ts(x,y)=\left(y,x\right),\;
s^2(x,y)=\left(\frac{1}{y},\frac{x}{y}\right).
\end{eqnarray*}
To compute the irreducible representations of $G$, we remark that $S_3$ acts on the group of irreducible characters
\[
\Xi=\mathrm{Hom}(A,K^\times)=
\left\{
\chi_{\kappa,\lambda}:0\leq\kappa,\lambda\leq n-1
\right\}
\]
of $A=\mathbb{Z}/n\mathbb{Z}\times\mathbb{Z}/n\mathbb{Z}$ via
\[g\chi_{\kappa,\lambda}(\sigma_{\alpha,\beta})=\chi_{\kappa,\lambda}(g^{-1}\sigma_{\alpha,\beta} g),\text{ where }\chi_{\kappa,\lambda}(\sigma_{\alpha,\beta})=\zeta^{\alpha\kappa+\beta\lambda}\text{ and }g\in S_3.
\]
This action gives all the irreducible representations of $G$ by the result of \cite[section 8.2]{SerreLinear}:
\begin{proposition}\label{prop:irr-reps-semidirect}
Let $\Xi$ be the group of irreducible characters of $A=\mathbb{Z}/n\mathbb{Z}\times\mathbb{Z}/n\mathbb{Z}$. For each $\chi_{\kappa,\lambda}\in\Xi$, we write $H_{\kappa,\lambda}$ for the stabilizer subgroup of $S_3$ with respect to $\chi_{\kappa,\lambda}$ and define $G_{\kappa,\lambda}=H_{\kappa,\lambda}\cdot A$. Let $\rho:H_{\kappa,\lambda}\rightarrow  \mathrm{GL}(V_{\kappa,\lambda})$ be an irreducible representation of $H_{\kappa,\lambda}$ and let $\widetilde{\rho}$ be the induced representation of $G_{\kappa,\lambda}$
\[
\xymatrix{
 G_{\kappa,\lambda} \ar[r]^-{\pi} \ar@/_1.0pc/[rr]_{\tilde{\rho}}& G_{\kappa,\lambda}/A=H_{\kappa,\lambda}  \ar[r]^-{\rho} & \mathrm{GL}(V_{\kappa,\lambda}).
 } 
 \]
 Then $\theta_{\kappa,\lambda,\rho}=\chi_{\kappa,\lambda}\otimes\widetilde{\rho}$ is an irreducible representation of $G$ and all irreducible representations of $G$ can be obtained in this manner, up to isomorphism.
\end{proposition}
We proceed with the explicit description of the action of $S_3$ on $\Xi$. We compute:
{\renewcommand{\arraystretch}{1.5}
\begin{equation}\label{actionS3onXi}
\begin{array}{|c|c|c|c|c|}
\hline
g\in S_3 &g^{-1} \sigma_{\alpha,\beta} g(x,y)  & g^{-1} \sigma_{\alpha,\beta} g & \chi_{\kappa,\lambda}\big(g^{-1} \sigma_{\alpha,\beta} g\big) &
g\cdot\chi_{\kappa,\lambda}
\\
\hline
s &  \left(\zeta^{\beta- \alpha} x,\zeta^{-\alpha} y\right)&\sigma_{\beta-\alpha,-\alpha} & \zeta^{\alpha(-\kappa-\lambda)+\beta\kappa} 
 & \chi_{-\kappa- \lambda,\kappa}
\\
\hline
s^2 & \left(\zeta^{-\beta} x,  \zeta^{\alpha-\beta} y\right) &\sigma_{-\beta,\alpha-\beta} &  \zeta^{\alpha\lambda+\beta(-\kappa-\lambda)} 
 & \chi_{\lambda,-\kappa- \lambda}
\\
\hline
t &   \left(\zeta^{-\alpha} x,\zeta^{\beta-\alpha} y\right)&\sigma_{-\alpha,\beta-\alpha} &  \zeta^{\alpha(-\kappa-\lambda)+\beta\lambda} 
 &\chi_{-\kappa -\lambda,\lambda}
\\
\hline
ts & \left(\zeta^{\beta} x,\zeta^{\alpha} y\right) &\sigma_{\beta,\alpha} &  \zeta^{\alpha\lambda+\beta\kappa} 
 & \chi_{\lambda,\kappa} 
\\
\hline
st &  \left(\zeta^{\alpha- \beta} x, \zeta^{- \beta} y\right)&\sigma_{\alpha-\beta,-\beta} & \zeta^{\alpha\kappa+\beta(-\kappa-\lambda)}
 &
\chi_{\kappa,-\kappa-\lambda }
\\
\hline
\end{array}
\end{equation}
}
\noindent Thus, a character $\chi_{\kappa,\lambda}\in \Xi$ is fixed by $s$ or $s^2$ if and only if $\kappa=\lambda=0$ or $\kappa=\lambda=\frac{n}{3}$ or $\kappa=\lambda=\frac{2n}{3}$when $3\mid n$. It is fixed by $t$ if and only if $\lambda=-2 \kappa$, it is fixed by $ts$ if and only if $\lambda=\kappa$ and it is fixed by $st$ if and only if $\kappa=-2 \lambda $. The trivial character is fixed by all $S_3$. This is summarized in the following table
{\renewcommand{\arraystretch}{1.5}
\[
\begin{array}{|c|c|c|c|}
\hline
\chi_{\kappa,\lambda}\in\Xi & \text{ Stabilizer }  H_{\kappa,\lambda}& \text{Orbit} & \text{Condition}
\\
\hline
\chi_{0,0} & S_3 & \{\chi_{0,0}\} & 
\\
\hline
\chi_{\frac{n}{3},\frac{n}{3}} & S_3 & \{\chi_{\frac{n}{3},\frac{n}{3}}\} &
3\mid n\\
\hline
\chi_{\frac{2n}{3},\frac{2n}{3}} & S_3 & \{\chi_{\frac{2n}{3},\frac{2n}{3}}\} &
3\mid n\\
\hline
\chi_{\kappa,\kappa} & \langle ts \rangle &
 \{ \chi_{\kappa,\kappa},\chi_{\kappa, -2 \kappa},\chi_{-2 \kappa,\kappa} \} 
 & \kappa\neq 0,\;\kappa\neq \frac{n}{3},\;\kappa\neq \frac{2n}{3}
\\
\hline
\chi_{\kappa,-2 \kappa} & \langle t \rangle & 
\{ \chi_{\kappa,\kappa},\chi_{\kappa, -2 \kappa},\chi_{-2 \kappa,\kappa} \} & \kappa \neq 0,\;\kappa\neq \frac{n}{3},\;\kappa\neq \frac{2n}{3}
\\
\hline
\chi_{-2 \kappa,\kappa} & \langle st \rangle & 
\{ \chi_{\kappa,\kappa},\chi_{\kappa, -2 \kappa},\chi_{-2 \kappa,\kappa} \} & \kappa\neq 0,\;\kappa\neq \frac{n}{3},\;\kappa\neq \frac{2n}{3}
\\
\hline
\chi_{\kappa,\lambda} & \langle 1 \rangle  & 
\{\chi_{\kappa,\lambda},
\chi_{-\kappa-\lambda,\kappa},
\chi_{\lambda,-\kappa - \lambda},
\chi_{-\kappa- \lambda	,\lambda},
\chi_{\lambda,\kappa},
\chi_{\kappa,-\lambda- \kappa	}
\}
& \kappa\neq \lambda, \lambda \neq - 2 \kappa, \kappa\neq -2 \lambda	
\\
\hline
\end{array}
\]
}
\noindent
A minimal set of representatives of the orbits is given by the 
elements $\chi_{0,0}$, $\chi_{\kappa,\kappa}$ for $\kappa\neq 0$ (and $\kappa\neq\frac{n}{3}$, if applicable) and $\chi_{\kappa,\lambda}$ for $\kappa\neq \lambda, \lambda \neq - 2 \kappa, \kappa\neq -2 \lambda$. Writing $A=\left(\mathbb{Z}/n\mathbb{Z}\times \mathbb{Z}/n\mathbb{Z}\right)$ as before, for each representative we have 
{\renewcommand{\arraystretch}{1.5}
\[
\begin{array}{|c|c|c|}
\hline
\text{Representative} & \text{Stabilizer } H_{\kappa,\lambda} & G_{\kappa,\lambda}=A \cdot H_{\kappa,\lambda}
\\
\hline
 \chi_{0,0},\chi_{\frac{n}{3},\frac{n}{3}},\chi_{\frac{2n}{3},\frac{2n}{3}} & S_3 & A\rtimes S_3 \\
\hline
\chi_{\kappa,\kappa} & \langle ts \rangle & A \rtimes \langle ts \rangle
\\
\hline
\chi_{\kappa,\lambda} & \langle 1 \rangle & A \\
\hline
\end{array}
\]
}
Every character $\chi_{\kappa,\lambda}\in \Xi$ of can be extended to a character of $G_{\kappa,\lambda}$ by defining $\chi_{\kappa,\lambda}(a g)=\chi_{\kappa,\lambda}(a)$, for $ a\in A, g \in H_{\kappa,\lambda}$.
We proceed with the irreducible representations of the stabilizers $H_{\kappa,\lambda}$. Recall that $S_3$ has three irreducible representations denoted by $\rho_\mathrm{triv},\rho_\mathrm{sgn}$ and $\rho_\mathrm{stan}$. The respective character table is given by:\\
\begin{center}
{\renewcommand{\arraystretch}{1.5}
\begin{tabular}{ |c|r|r|r|r|r|r| } 
 \hline
 &$1$ & $s$ & $t$ & $st$ & $ts$ & $s^2$ \\ 
 \hline
 $\chi_{\text{triv}}$ &$1$ & $1$ & $1$ & $1$ & $1$ & $1$\\ 
 \hline
 $\chi_{\text{sgn}}$ &$1$ & $1$ & $-1$ & $-1$ & $-1$ & $1$ \\ 
 \hline
$\chi_{\text{stan}}$ &$2$ & $-1$ & $0$ & $0$ & $0$ & $-1$ \\ 
 \hline
\end{tabular}
}
\end{center}
Every irreducible representation $\rho:H_{\kappa,\lambda}\rightarrow \mathrm{GL}(V)$ gives rise 
to an irreducible representation $\tilde{\rho}$ of $G_{\kappa,\lambda}$ by considering the composition
\[
\xymatrix{
 G_{\kappa,\lambda} \ar[r]^-{\pi} \ar@/_1.0pc/[rr]_{\tilde{\rho}}& G_{\kappa,\lambda}/A=H_{\kappa,\lambda}  \ar[r]^-{\rho} & \mathrm{GL}(V).
 } 
 \]
By Proposition \ref{prop:irr-reps-semidirect}, the irreducible representations of $G$ are given by $\theta_{\kappa,\lambda,\rho}=\mathrm{Ind}_{G_{\kappa,\lambda}}^{G} \chi_{\kappa,\lambda}\otimes \tilde{\rho}$:\\

\noindent $\bullet$
The stabilizer $H_{0,0}$ of $\chi_{0,0}$ equals the group $S_3$. Thus we have the representations 
$
\theta_{0,0,\rho_{\text{triv}}}, \theta_{0,0,\rho_{\mathrm{sgn}}},\theta_{0,0,\rho_{\mathrm{stan}}}
$
with corresponding characters
\[
\chi_{0,0,\text{triv}}(\sigma_{\alpha, \beta} g)=\chi_{\text{triv}}( g)
,\;
 \chi_{0,0,\mathrm{sgn}}(\sigma_{\alpha, \beta} g)= \chi_{\mathrm{sgn}}( g)
,\;
\chi_{0,0,\mathrm{stan}}(\sigma_{\alpha, \beta} g)=\chi_{\mathrm{stan}}( g),
\]
for $\sigma_{\alpha,\beta}\in A$ and $g \in S_3$.\\

\noindent $\bullet$
The stabilizer  $H_{\kappa,\kappa}=\langle ts \rangle$ of $\chi_{\kappa,\kappa}$ for $\kappa\neq0 ,\kappa\neq \frac{n}{3},\kappa\neq \frac{2n}{3}$ has two one dimensional representations, $\rho_{\text{triv}}$ and $\rho_\text{sgn}$, which induce the representations $\theta_{\kappa,\kappa,\rho_\text{tr}}$ and $\theta_{\kappa,\kappa,\rho_\text{sgn}}$ respectively. To compute the respective characters, we write the arbitrary element of $G$ as $bx$ for $b\in A$ and $x\in S_3$. For any element $y=ag\in G$ with $a\in A$ and $g\in S_3$ we have
\[
y^{-1}bxy=g^{-1}a^{-1}bxag=(g^{-1}a^{-1}bxax^{-1}g) g^{-1}xg.
\]
We remark that $g^{-1}a^{-1}bxax^{-1}g \in A$: indeed, since $a\in A$ and $x\in S_3$, we have that $xax^{-1}\in A$ and thus $a^{-1}bxax^{-1} \in A$ as well so that $g^{-1}a^{-1}bxax^{-1}g \in A$. We conclude that the element $y^{-1} bx y \in A \cdot\langle ts \rangle$ if and only if $g^{-1} x g \in \langle ts \rangle$. If $x\in\{s,s^2\}$, then $g^{-1}xg$ has order $3$ and therefore $g^{-1}xg\notin\langle ts \rangle$. If $x=1$ then trivially $g^{-1}xg=1$. Finally, for each  $x\in S_3-\langle s \rangle=\left\{ t,ts,st\right\}$ there are two elements 
$g\in S_3$ such that $g^{-1} x g =ts$:
\[
{\renewcommand{\arraystretch}{1.5}
\begin{array}{|c||c|c|c|}
\hline
x & t & ts & st 
\\
\hline
g \text{ such that }g^{-1} x g =ts& s^2,st & 1,ts & s,t \\
\hline
\end{array}
}
\]
For $\rho\in\{\rho_\mathrm{triv},\rho_\mathrm{sgn}\}$, the corresponding characters are:
\begin{align*}
\chi_{\kappa,\kappa,\rho}(b x) &= \frac{1}{2n^2} 
\sum_{y\in G
\atop
g^{-1} x g \in G_{\kappa,\kappa}
} 
\chi_{\kappa,\kappa}(y^{-1} bx y) \otimes \tilde{\rho}(y^{-1} bx y)
\\
&=
\begin{cases}
\displaystyle
\frac{1}{2n^2}
\sum_{a\in A, g\in S_3 \atop
g^{-1} x g \in \langle ts \rangle 
}
\chi_{\kappa,\kappa}(  g^{-1} a^{-1} b x a g) \otimes \rho(g^{-1} x g) & \text{, if } x \neq s,s^2\\
 0 & \text{, if } x= s,s^2
\end{cases}
\end{align*}
\begin{lemma}
Let $x\in\{1,t,ts,st \}$ and let $g\in S_3$ be such that $g^{-1}xg\in \langle ts\rangle$. Then
\[
\chi_{\kappa,\kappa}(g^{-1}a^{-1}bxag)=g\chi_{\kappa,\kappa}(b).
\]
\end{lemma}
\begin{proof}
Observe that
\begin{equation*}
\chi_{\kappa,\kappa}\left(g^{-1}a^{-1}bxag \right)=\chi_{\kappa,\kappa}\left(g^{-1}a^{-1}bxax^{-1}gg^{-1}xg \right)=\chi_{\kappa,\kappa}\left(g^{-1}a^{-1}bxax^{-1}g \right)=g\chi_{\kappa,\kappa}\left(a^{-1}bxax^{-1} \right)
\end{equation*}
Since $a^{-1},b,xax^{-1}\in A$ and $A$ is abelian, we conclude that
\[
\chi_{\kappa,\kappa}\left(g^{-1}a^{-1}bxag \right)
=\left(g\chi_{\kappa,\kappa}\left(a^{-1}xax^{-1} \right)\right)\left(g\chi_{\kappa,\kappa}\left(b \right)\right).
\]
Writing $a=\sigma_{\alpha,\beta}$, we have the following table:
\[
{\renewcommand{\arraystretch}{1.5}
\begin{array}{|c|c|c|c|c|c|}
\hline
x   & xax^{-1}  & a^{-1}xax^{-1} &  g & g \chi_{\kappa,\kappa}  &  g \chi_{\kappa,\kappa} (xax^{-1}\cdot a^{-1}) 
\\
\hline
t  &  \sigma_{-\alpha,\beta-\alpha} & \sigma_{-2\alpha,-\alpha} & s^2, st& 
\chi_{\kappa,-2 \kappa} &
1
\\
\hline
ts & \sigma_{\beta,\alpha} &  \sigma_{\beta-\alpha,\alpha-\beta}& 1,ts  & 
\chi_{\kappa,\kappa} & 1
\\
\hline
st  & \sigma_{\alpha-\beta,-\beta} &  \sigma_{-\beta,-2\beta}& s,t & \chi_{-2 \kappa,\kappa	} &
1
\\
\hline
\end{array}
}
\]
\end{proof}
We conclude that
\begin{equation*}
\chi_{\kappa,\kappa,\rho}(bx)=
\begin{cases}
\displaystyle
\frac{1}{2n^2}
\sum_{a\in A, g\in S_3 
}
\chi_{\kappa,\kappa}(  b) \otimes \rho(1)&, \text{ if }x=1\\
\displaystyle
\frac{1}{2n^2}
\sum_{a\in A \atop
g^{-1} x g =ts  
}
g\chi_{\kappa,\kappa}
\left(  bx\right)
\otimes \rho(ts)&, \text{ if }x=t,ts,st\\
0&, \text{ if }x=s,s^2
\end{cases}
\end{equation*}
and, explicitly, writing $b=\sigma_{\alpha,\beta}$, we arrive at
\[
\chi_{\kappa,\kappa,\rho}(\sigma_{\alpha,\beta} x) =
\begin{cases}
\zeta^{\kappa (\alpha+\beta)} +
  \zeta^{\kappa (\alpha -2 \beta}) +
  \zeta^{\kappa (-2\alpha  +\beta)  },
  & \text{ if } x =1 \\
\zeta^{\kappa (\alpha+\beta)} \rho(x), & \text{ if } x=ts
\\
\zeta^{\kappa (\alpha-2 \beta)} \rho(x), & \text{ if } x=t
\\
\zeta^{\kappa (-2 \alpha +\beta)} \rho(x),   & \text{ if } x=st
 \\
0,  & \text{ if }x=s,s^2
\end{cases}
\]

\noindent $\bullet$ Finally the generic $\chi_{\kappa,\lambda}$ 
has trivial stabilizer which admits only the trivial representation, so we have a unique representation $
\theta_{\kappa,\lambda,\rho_\text{triv}} = \mathrm{Ind}_{A}^G \chi_{\kappa,\lambda}$, with character
\begin{align*}
\chi_{\kappa,\lambda,\rho_{\text{triv}}}(b) &=\frac{1}{n^2} 
\sum_{g\in G
\atop
g^{-1} bg \in A
} \chi_{\kappa,\lambda}(gbg^{-1})\\
&=
\begin{cases}  0 & \text{ if } b \not\in A \\
\chi_{\kappa,\lambda}(b)+
\chi_{-\kappa-\lambda,\kappa}(b)+
\chi_{\lambda,-\kappa - \lambda}(b)+
\chi_{-\kappa- \lambda	,\lambda}(b)+
\chi_{\lambda,\kappa}(b)+
\chi_{\kappa,-\lambda- \kappa	}(b)
& \text{ if } b \in A
\end{cases}
\end{align*}
We have proved the following:\\

\begin{mdframed}[backgroundcolor=blue!6]
\begin{proposition}\label{prop:Irreps}
The irreducible representations of the group $G$ are given in the following table
\[
{\renewcommand{\arraystretch}{2}
\!\!\!\!
\begin{array}{|c|c|c|}
\hline
\textbf{Representation} 
&
\textbf{ Degree }
&
\textbf{ Character } \chi(\sigma_{\alpha,\beta}x), \text{ where }  x\in S_3 
\\
\hline
\theta_{\frac{\nu n}{3},\frac{\nu n}{3},\rho} & 1 &  \zeta^{\frac{\nu n}{3}(\alpha+\beta)}\chi_\rho(x)
\\
\hline
\theta_{\frac{\nu n}{3},\frac{\nu n}{3},\rho_\mathrm{stan}}
&  2 &\zeta^{\frac{\nu n}{3}(\alpha+\beta)}\chi_\mathrm{stan}(x)
\\
\hline
\theta_{\kappa ,\kappa ,\rho} &  3 & 
\begin{cases}
\zeta^{\kappa (\alpha+\beta)} +
  \zeta^{\kappa (\alpha -2 \beta)} +
  \zeta^{\kappa (\beta-2\alpha  )  }
  &, \text{ if } x =1 \\
\zeta^{\kappa (\alpha+\beta)} \rho(x) &, \text{ if } x=ts
\\
\zeta^{\kappa (\alpha-2 \beta)} \rho(x) &, \text{ if } x=t
\\
\zeta^{\kappa (\beta-2 \alpha)} \rho(x)   &, \text{ if } x=st
 \\
0  &, \text{ if } x=s,s^2
\end{cases}
\\
\hline
\theta_{\kappa ,\lambda, \rho_\mathrm{triv} } & 6  & 
\begin{cases}  
\left(\begin{array}{l}
\zeta^{\kappa \alpha+ \lambda \beta} +
\zeta^{-(\kappa +\lambda) \alpha +\kappa \beta }+
\zeta^{\lambda \alpha -(\kappa+\lambda) \beta}+
\\
\zeta^{\lambda \alpha + \kappa \beta}+
\zeta^{-(\kappa+\lambda)\alpha +\lambda \beta}+
\zeta^{\kappa \alpha - (\kappa+\lambda) \beta}
\end{array}\right) &, \text{ if } x = 1 \\
0
&, \text{ if } x\neq 1
\end{cases}
\\
\hline
\end{array}
}
\]
where $\nu\in\{0,1,2\},\;\rho\in\{\rho_\mathrm{triv},\rho_\mathrm{sgn}\}$ and the representations corresponding to $\kappa,\lambda\in\{\frac{n}{3},\frac{2n}{3}\}$ appear only when $3\mid n$.
\end{proposition}
\end{mdframed}

\section{Character tables for $m$-differentials}\label{sec:Chars}

Let $\Omega_{F_n}$ denote the sheaf of holomorphic differentials on the Fermat curve $F_n$. More generally, for $m\geq 1$, we write $\Omega_{F_n}^{\otimes m}$ for the sheaf of holomorphic $m$-differentials. The global sections $H^0(F_n,\Omega_{F_n}^{\otimes m})$ form a vector space over the  ground field $K$; its dimension is given by the Riemann-Roch Theorem:
\begin{align*}
\dim_{K}H^0(F_n,\Omega_{F_n}^{\otimes m})=
\begin{cases}
g=\displaystyle\frac{(n-1)(n-2)}{2},& \text{ if }m=1,\\\\
(2m-1)(g-1)=(2m-1)\displaystyle\frac{n(n-3)}{2},& \text{ if }m\geq 2.\\
\end{cases}
\end{align*}
We proceed with the description of an explicit basis:
\begin{proposition}\label{prop:bases-diffs}
A $K$-basis for  $V_m=H^0(F_n,\Omega_{F_n}^{\otimes m})$ is given by
\[
\mathbf{b}_m=
\begin{cases}
\left\{
\displaystyle\frac{x^iy^j}{y^{n-1}}dx:
0\leq i,j,i+j\leq n-3
\right\}\text{, if } m=1\\\\
\left\{
\displaystyle\frac{x^iy^j}{y^{m(n-1)}}dx^{\otimes m}:
0\leq i\leq n-1,\;0\leq j, i+ j\leq m(n-3)
\right\}\text{, if } m\geq2.
\end{cases}
\]
\end{proposition}
\begin{proof}
By definition, all differentials that appear in $\mathbf{b}_m$, for $m\geq 1$, are holomorphic and linearly independent and a simple counting argument verifies that they have the correct cardinality. 
\end{proof}
The above basis is not suitable for computations when $m \geq 2$,  since it is not symmetric in $i,j$.
\begin{definition}\label{def:W_m-I_m}
Let $E_M\subseteq\mathbb{Z}^2$ denote the triangle with vertices $(0,0),(M,0),(0,M)$, that is
\[
E_M=\{(i,j)\in \mathbb{Z}^2: 0\leq i,j,i+j\leq M\}.
\]
We will denote by $W_m$ the $K$-vector space generated by the set
\begin{equation}\label{eq:basis-W_m}
\left\{
w_{i,j}^{(m)}:=
\frac{x^iy^j}{y^{m(n-1)}}dx^{\otimes m}:
(i,j)\in E_{m(n-3)}
\right\}.
\end{equation}
and by $I_m$ its $K$-subspace generated by the set
\begin{equation}\label{eq:basis-I_m}
\left\{
\pi_{i,j}^{(m)}:=\frac{x^i y^j}{y^{m(n-1)}}dx^{\otimes m}+\frac{x^{i+n} y^j}{y^{m(n-1)}}dx^{\otimes m}+\frac{x^i y^{j+n}}{y^{m(n-1)}}dx^{\otimes m}\in W_m
:(i,j)\in E_{m(n-3)}
\right\}.
\end{equation}
\end{definition}
Observe that if $m=1$, then $W_1=H^0(F_n,\Omega_{F_n})$ whereas for $m\geq 2$ we have that
\[
W_m/I_m\cong H^0(F_n,\Omega_{F_n}^{\otimes m})
\]
since $\pi_{i,j}^{(m)}=w_{i,j}^{(m)}\left(x^n+y^n+1\right)=0\in H^0(F_n,\Omega_{F_n}^{\otimes m})$ by the curve's defining equation. 
The reader may observe that this decomposition of $H^0(F_n,\Omega_{F_n}^{\otimes m})$ is closely related to the classic result of M. Noether, F. Enriques and K. Petri on the canonical ideal of non-hyperelliptic curves: indeed the graded ring $\bigoplus W_m$ is isomorphic to the quotient of ${\rm Sym}\left(H^0(F_n,\Omega_{F_n}) \right)$ by some binomial relations, whereas the elements of $I_m$ are the missing generators for the kernel of the canonical map
\[
{\rm Sym}\left(H^0(F_n,\Omega_{F_n}) \right)\twoheadrightarrow \bigoplus H^0(F_n,\Omega_{F_n}^{\otimes m}).
\]
For more details on the explicit construction in the case of Fermat curves see \cite{190910282}, an application of a more general technique given in \cite{charalambous2019relative}. We proceed with the description of the action of $G$ on $W_1=H^0(F_n,\Omega_{F_n})$:
\begin{proposition}\label{prop:action-differentials}
Let $\omega_{i,j}:=w_{i,j}^{(1)}=\displaystyle\frac{x^iy^j}{y^{n-1}}dx\in\mathbf{b}_1$. Then
\[
{\renewcommand{\arraystretch}{1.5}
\begin{array}{|c|c|c|c|c|c|c|}
\hline
g\in G &\sigma_{\alpha,\beta} & s & t & ts & st & s^2 \\
\hline
g\cdot\omega_{i,j} & \zeta^{\alpha(i+1)+\beta (j+1)}\omega_{i,j}&\omega_{n-3-(i+j),i}   & -\omega_{n-3-(i+j),j}   & -\omega_{j,i}  &-\omega_{i,n-3-(i+j)}   &\omega_{j,n-3-(i+j)}  \\
\hline
\end{array}
}
\]
\end{proposition}
\begin{proof}
For $\sigma_{\alpha,\beta}\in\mathbb{Z}/n\mathbb{Z}\times\mathbb{Z}/n\mathbb{Z}$ we have
\begin{eqnarray*}
\sigma_{\alpha,\beta}(\omega_{i,j})&=&\sigma_{\alpha,\beta}\left(\frac{x^iy^j}{y^{n-1}}dx\right)
=\frac{\left(\zeta^\alpha x\right)^i\left(\zeta^\beta y\right)^j}{\left(\zeta^\beta y\right)^{n-1}}d\left(\zeta^\alpha x\right)
=\zeta^{\alpha(i+1)+\beta (j+1)}\omega_{i,j}
\end{eqnarray*}
For the element $s\in S_3$ of order $2$ we have
\begin{eqnarray*}
s(\omega_{i,j})&=&s\left(\frac{x^iy^j}{y^{n-1}}dx\right)=\left(\frac{y}{x}\right)^i\left(\frac{1}{x}\right)^jx^{n-1}d\left(\frac{y}{x}\right).\\
\end{eqnarray*}
Note that differentiating the Fermat equation $x^n+y^n+1=0$ gives
$dy=-\frac{x^{n-1}}{y^{n-1}}dx$ 
which in turn gives
\begin{eqnarray*}
s(\omega_{i,j})&=&\frac{y^i}{x^{i+j}}x^{n-1}\frac{1}{x^2 y^{n-1}}dx=\frac{x^{n-3-(i+j)}y^i}{ y^{n-1}}dx=\omega_{n-3-(i+j),i}.
\end{eqnarray*}
The computation of the other elements follows in the same way. 
\end{proof}
When $m\geq 2$, for each $w_{i,j}^{(m)}$ there exist basis elements $\omega_{i_1,j_1},\cdots, \omega_{i_m,j_m}\in H^0(F_n,\Omega_{F_n})$ such that
\begin{equation}\label{eq:W_1-suffices}
w_{i,j}^{(m)}=\omega_{i_1,j_1}\otimes\cdots\otimes \omega_{i_m,j_m}
\end{equation}
where $i=i_1+\cdots i_m,\;j=j_1+\cdots j_m$. This remark, combined with Proposition \ref{prop:action-differentials} gives directly an explicit description of the action of $G$ on $W_m$ and $I_m$ respectively.
\begin{proposition}\label{prop:action-m-differentials}
Let $w_{i,j}^{(m)}$ be as in eq (\ref{eq:basis-W_m}) and $\pi_{i,j}^{(m)}$ be as in eq (\ref{eq:basis-I_m}). Then
\begin{equation*}
{\renewcommand{\arraystretch}{2}
\begin{array}{|c|c|c|}
\hline
g\in G &g\cdot w^{(m)}_{i,j}&g\cdot \pi^{(m)}_{i,j}   \\
\hline
 \sigma_{\alpha,\beta}& \zeta^{\alpha(i+m)+\beta (j+m)}w^{(m)}_{i,j}& \zeta^{\alpha(i+m)+\beta (j+m)}\pi^{(m)}_{i,j} \\
\hline
s& w^{(m)}_{m(n-3)-(i+j),i} & \pi^{(m)}_{m(n-3)-(i+j)-n,i} \\
\hline
s^2 & w^{(m)}_{j,m(n-3)-(i+j)} & \pi^{(m)}_{j,m(n-3)-(i+j)-n}\\
\hline
t &  (-1)^m w^{(m)}_{m(n-3)-(i+j),j} &  (-1)^m \pi^{(m)}_{m(n-3)-(i+j)-n,j}
\\
\hline
ts & (-1)^m w^{(m)}_{j,i} & (-1)^m \pi^{(m)}_{j,i}\\
\hline
st & (-1)^m w^{(m)}_{i,m(n-3)-(i+j)}& (-1)^m \pi^{(m)}_{i,m(n-3)-(i+j)-n}  \\
\hline
\end{array}
}
\end{equation*}
\end{proposition}
We proceed with an explicit description of the character tables of the representations $G\rightarrow GL\left(W_m\right)$ and $G\rightarrow GL\left(I_m\right)$:
\begin{proposition}\label{prop:Wm-Im-chars}
The characters $\chi_{W_m}$ and $\chi_{I_m}$ of the representations $G\rightarrow GL\left(W_m\right)$ and $G\rightarrow GL\left(I_m\right)$ respectively are given in the following table
\begin{equation*}\label{abCharH0}
{\renewcommand{\arraystretch}{2.75}
\begin{array}{|c|c|c|}
\hline
g\in G &\text{ Character }\chi_{W_m}(g) &\text{ Character }\chi_{I_m}(g) \;\; m >1  \\
\hline
 \sigma_{\alpha,\beta}& \displaystyle\sum_{i=0}^{m(n-3)}\sum_{j=0}^{m(n-3)-i} \zeta^{\alpha(i+m)+\beta (j+m)}
&
\displaystyle\sum_{i=0}^{m(n-3)-n}\sum_{j=0}^{m(n-3)-n-i} \zeta^{\alpha(i+m)+\beta (j+m)}
 \\
\hline
\sigma_{\alpha,\beta}s & \displaystyle\begin{cases}
\zeta^{(\alpha+\beta)\frac{mn}{3}}, & \text{ if } 3\;|\;n\text{ or }3\;|\;m\\
0, & \text{ if } 3\nmid n\text{ and } 3\nmid m
\end{cases}
&
\begin{cases}
\zeta^{(\alpha+\beta)\frac{n(m-1)}{3}}, & \text{ if } 3\;|\;n(m-1)\\
0, &\text{ otherwise }
\end{cases} 
\\
\hline
\sigma_{\alpha,\beta}s ^2 &
\displaystyle\begin{cases}
\zeta^{(\alpha+\beta)\frac{mn}{3}}, & \text{ if } 3\;|\;n\text{ or }3\;|\;m\\
0, & \text{ if } 3\nmid n\text{ and } 3\nmid m
\end{cases} 
&
\begin{cases}
\zeta^{(\alpha+\beta)\frac{n(m-1)}{3}}, & \text{ if } 3\;|\;n(m-1)\\
0, &\text{ otherwise }
\end{cases} 
\\
\hline
\sigma_{\alpha,\beta}t & \displaystyle
(-1)^m \sum_{i=0}^{ \lf \frac{m(n-3)}{2} \rf } \zeta^{ (\alpha-2\beta)(i+m)}
&
\displaystyle
(-1)^m \sum_{i=0}^{ \lf \frac{m(n-3)-n}{2} \rf } \zeta^{ (\alpha-2\beta)(i+m)}
\\
\hline
\sigma_{\alpha,\beta}ts & \displaystyle
 (-1)^m\sum_{i=0}^{ \lf \frac{m(n-3)}{2} \rf } \zeta^{(\alpha+\beta)(i+m)} 
&
 \displaystyle
 (-1)^m\sum_{i=0}^{ \lf \frac{m(n-3)-n}{2} \rf } \zeta^{(\alpha+\beta)(i+m)} 
\\
\hline
\sigma_{\alpha,\beta}st & \displaystyle (-1)^m \sum_{i=0}^{ \lf \frac{m(n-3)}{2} \rf } \zeta^{ (\beta-2\alpha)(i+m)}
&
\displaystyle (-1)^m \sum_{i=0}^{ \lf \frac{m(n-3)-n}{2} \rf } \zeta^{ (\beta-2\alpha)(i+m)}
\\
\hline
\end{array}
}
\end{equation*}
\end{proposition}
\begin{proof}
By eq. (\ref{eq:W_1-suffices}), it suffices to compute explicitly the $m=1$ case: by Proposition \ref{prop:action-differentials}, the matrix $\rho(\sigma_{\alpha,\beta})$ is diagonal with trace
\[
\chi_{W_1}(\sigma_{\alpha,\beta})=\sum_{0\leq i+j\leq n-3} \zeta^{\alpha(i+1)+\beta (j+1)}.
\]
Similarly, the diagonal entries of the matrix $\rho(s)$ are either $0$ or $1$. The number of non-zero such diagonal entries equals to the number of pairs $(i,j)$ that satisfy the relations:
\begin{equation*}
0\leq i+j\leq n-3,\; i=n-3-(i+j),\;i=j
\end{equation*}
which gives $3i=n-3$. Thus, we have a unique non-zero diagonal entry if only if $3\mid n-3$, i.e. if $3\mid n$ and this gives the formula for $\chi_{W_1}(s)$. The formula for $\chi_{W_1}(\sigma_{\alpha,\beta}s)$ is obtained by substituting $i=j=\frac{n-3}{3}$ in $\zeta^{\alpha(i+1)+\beta(j+1)}$. \\

We proceed with computing $\chi_{W_1}(t)$ and $\chi_{W_1}(\sigma_{\alpha,\beta}t)$: again by Proposition \ref{prop:action-differentials} we have that the diagonal entries of the matrix $\rho(t)$ are either $0$ or $-1$.  The number of non-zero such diagonal entries, is given by the number of pairs $(i,j)$ that satisfy the relations:
\begin{equation*}
0\leq i+j\leq n-3,\; i=n-3-(i+j).
\end{equation*}
Thus, for each value of $i$ we have that $j=(n-3)-2i$.  Since both $i$ and $j$ must be non-negative, the number of pairs $(i,j)$ satisfying the above relations is equal to the number of unique $i$-values that satisfy
\[
0\leq i\leq \left\lfloor\frac{n-3}{2}\right\rfloor.
\]
which gives the formula for $\chi_{W_1}(t)$. The formula for $\chi_{W_1}(\sigma_{\alpha,b} t)$ follows from substituting $j=(n-3)-2i$ in $-\zeta^{\alpha(i+1)+\beta(j+1)}$.

The fact that characters are invariant under conjugation gives 
$\chi_{W_1}(s^2)=\chi_{W_1}(s)$ and $\chi_{W_1}(st)=\chi(ts)=\chi_{W_1}(t)$. The three remaining cases, i.e. $\chi_{W_1}(\sigma_{\alpha,\beta}s^2),\chi(\sigma_{\alpha,\beta}st)$ and $\chi_{W_1}(\sigma_{\alpha,\beta}ts)$ are obtained by substituting the respective equations for $i,j$ into the formula for $\chi_{W_1}(\sigma_{\alpha,\beta})$.

The formula for $\chi_{I_m}\left(\sigma_{\alpha,\beta}\right)$ is obtained by replacing the bounds $0\leq i,j,i+j\leq m(n-3)$ by the new bounds $0\leq i,j,i+j\leq m(n-3)-n$. The characters $\chi_{I_m}(s)=\chi_{I_m}(s^2)$ are determined as follows
\[
\pi_{i,j}=s\pi_{i,j}\Leftrightarrow \pi_{i,j}=s^2\pi_{i,j}\Leftrightarrow
\begin{cases}
i=j\\
j=m(n-3)-n-2i
\end{cases}
\Leftrightarrow
3i=m(n-3)-n=n(m-1)-3m.
\]
Thus, the characters $\chi_{I_m}(s)=\chi_{I_m}(s^2)$ are non-zero if and only if $3\mid n(m-1)-3m$ or equiv. $3\mid n(m-1)$. Proceeding with the characters $\chi_{I_m}(t)=\chi_{I_m}(ts)=\chi_{I_m}(st)$, Proposition \ref{prop:action-m-differentials} gives rise to the system
\[
\begin{cases}
2i=m(n-3)-n-j\\
0\leq i,j\leq m(n-3)-n
\end{cases}
\Leftrightarrow
0\leq 2i\leq m(n-3)-n
\Leftrightarrow
0\leq i\leq \left\lfloor\frac{m(n-3)-n}{2}\right\rfloor
\]
which implies that
\[
\chi_{I_m}(t)=\chi_{I_m}(ts)=\chi_{I_m}(st)=(-1)^m\left(\left\lfloor\frac{m(n-3)-n}{2}\right\rfloor+1\right)=(-1)^m\left\lfloor\frac{m(n-3)-n+2}{2}\right\rfloor.
\]
The formulas for $\chi_{I_m}(\sigma_{\alpha,\beta}g),\;g\in S_3$ follow directly by replacing the equations satisfied by $(i,j)$ into the formula for $\chi_{I_m}(\sigma_{\alpha,\beta})$.
\end{proof}
The above result and the isomorphism $W_m/I_m\cong H^0(F_n,\Omega_{F_n}^{\otimes m})$ allows to obtain the characters of the spaces $V_m=H^0(F_n,\Omega_{F_n}^{\otimes m})$:\\
\begin{mdframed}[backgroundcolor=blue!6]
\begin{theorem}\label{th:chars-m-diff}
The character for the space $H^0(F_n,\Omega_{F_n})$ is equal to the character of the space $W_1$, while the character for $H^0(F_n,\Omega_{F_n}^{\otimes m})$, for $m\geq 2$ equals $\chi_{W_m}-\chi_{I_m}$, where $\chi_{W_m},\chi_{I_m}$ are given in Proposition \ref{prop:Wm-Im-chars}.
\end{theorem}
\end{mdframed}

\section{Computing sums of roots of unity}\label{sec:IJ}
To determine explicitly the Galois module structure of the spaces $H^0(F_n,\Omega_{F_n}^{\otimes m})$ we will use the standard approach of computing the inner product of the irreducible characters of $G$, given in Proposition \ref{prop:Irreps}, with the characters $\chi_{W_m}-\chi_{I_m}$ given in Proposition \ref{prop:Wm-Im-chars}. These computations require finding closed formulas for two types of sums that involve $n$-th roots of unity:
\begin{proposition}\label{prop:general-IJ}
For $M\in\mathbb{N}$, let $E_M\subseteq\mathbb{N}^2$ be the triangle with vertices $(0,0),(0,M),(M,0)$. Then, for any $X,Y\in\Z$ we have that
\begin{eqnarray*}
I_M(X,Y)&:=&
\sum_{\alpha,\beta\in\Z/n\Z}\sum_{(i,j)\in E_M}\zeta^{\alpha(i+X)+\beta(j+Y)}\\
&=&
n^2\cdot\#\left\{(i,j)\in E_M:\;i\equiv -X\;{\rm mod}\;n \text{ and }j\equiv -Y\;{\rm mod}\;n \right\}\\\\
J_M(X)&:=&
\sum_{\alpha,\beta\in\Z/n\Z}
\sum_{i=0}^{\left\lfloor\frac{M}{2}\right\rfloor}\zeta^{(\alpha+\beta)(i+X)}
=
\sum_{\alpha,\beta\in\Z/n\Z}
\sum_{i=0}^{\left\lfloor\frac{M}{2}\right\rfloor}\zeta^{(\alpha-2\beta)(i+X)}
=
\sum_{\alpha,\beta\in\Z/n\Z}
\sum_{i=0}^{\left\lfloor\frac{M}{2}\right\rfloor}\zeta^{(\beta-2\alpha)(i+X)}\\
&=&n^2\cdot\#
\left\{
i\in\mathbb{N}:0\leq i\leq \left\lfloor\frac{M}{2}\right\rfloor\text{ and }i\equiv -X\;{\rm mod}\; n
\right\}.
\end{eqnarray*}
\end{proposition}
\begin{proof}
Observe that for fixed $(i,j)\in E_M$ we have that
\begin{align*}
\sum_{\alpha,\beta\in\Z/n\Z}\zeta^{\alpha(i+X)+\beta(j+Y)}=
\begin{cases}
n^2,&\text{ if }i\equiv -X\;{\rm mod}\;n \text{ and }j\equiv -Y\;{\rm mod}\;n\\
0,&\text{ otherwise }
\end{cases}
\end{align*}
and thus $I_M(X,Y)$ is either $0$ or equal to $n^2$ times the number of solutions of the system of congruences  $i\equiv -X\;{\rm mod}\;n \text{ and }j\equiv -Y\;{\rm mod}\;n$ which lie in $E_M$.\\

Similarly, for fixed $0\leq i\leq \lc\frac{M}{2}\rc$ we have that
\begin{align*}
\sum_{\alpha,\beta\in\Z/n\Z}
\zeta^{(\alpha+\beta)(i+X)}
=
\sum_{\alpha,\beta\in\Z/n\Z}
\zeta^{(\alpha-2\beta)(i+X)}
=
\sum_{\alpha,\beta\in\Z/n\Z}
\zeta^{(\beta-2\alpha)(i+X)}=
\begin{cases}
n^2,&\text{ if }i\equiv -X\;{\rm mod}\;n \\
0,&\text{ otherwise }
\end{cases}
\end{align*}
and thus $J_M(X)$ is either $0$ or equal to $n^2$ times the number of solutions of the congruence $i\equiv -X\;{\rm mod}\;n$ which satisfy $0\leq i\leq \lc\frac{M}{2}\rc$.
\end{proof}
Recall that our motivation for computing the quantities $I_M(X,Y)$ and $J_M(X)$ is to obtain the Galois module structure of the global sections of $m$-differentials for $m\geq 1$. Since the respective characters are given by the differences $\chi_{W_m}-\chi_{I_m}$ we need formulas for {\em differences} of these quantities for different values of $M$:
\begin{proposition}\label{prop:count-IJ}
Assume that $M+1-n \geq 0$.
With the notation of Proposition \ref{prop:general-IJ}, let
\begin{align*}
I(X,Y)&:=I_M(X,Y)-I_{M-n}(X,Y)\\
J(X)&:=J_M(X)-J_{M-n}(X).
\end{align*}
For an integer $\kappa$ we will denote by $v_\kappa$ the remainder of the division of $\kappa$ by $n$. 
We have:
\begin{align*}
I(X,Y)&=
n^2\left(
\lf\frac{M
+1
}{n}\rf+\alpha_{X,Y}
\right)
\text{ where }
\alpha_{X,Y}=
\begin{cases}
 -1 &\displaystyle\text{ if }
v_{-X} \geq  v_{M+1}
\text{ and }
v_{-X} -v_{M+1} +v_{-Y} \geq n
\\
 1 &\displaystyle\text{ if }
v_{-X} <  v_{M+1}
\text{ and }
n- v_{M+1}+ v_{-X}+v_{-Y} < n\\
 0 &\text{ otherwise }.
\end{cases} \\
J(X) &=
\lf
\frac
{
\lf 
\frac{M}{2}
\rf
-v_{-X}
}
{n}
\rf
-
\lc
\frac{
	\lf \frac{M-n}{2} \rf - v_{-X}+1
}
{
	n
}
\rc
\end{align*}
 If $M+1<n$ then since $I_\kappa=J_\kappa=0$
for $k<0$ we have $I(X,Y)=I_M(X,Y)$ and $J(X)=J_M(X)$, 
in this case
\[
I(X,Y)=
\begin{cases}
n^2 & \text{ if } v_{-X}+v_{-Y} \leq M \\
0 & \text{ otherwise}
\end{cases}
\]
and 
\[
J(X)= \lf
\frac
{
\lf 
\frac{M}{2}
\rf
-v_{-X}
}
{n}
\rf+1.
\]
Since if $M+1<n$ we have $\lf \frac{M+1}{n} \rf=0$ we can define $\alpha_{X,Y}\in \{0,1\}$ in this case as well by 
$\alpha_{X,Y}=\frac{1}{n^2}I(X,Y)$.
\end{proposition}
\begin{proof}
By Proposition \ref{prop:general-IJ}, $I(X,Y)$ is equal to $n^2$ times the number of solutions to the system
\begin{equation}\label{eq:system}
\begin{cases}
i\equiv -X\;{\rm mod}\;n \\
j\equiv -Y\;{\rm mod}\;n 
\end{cases}
\end{equation}
that lie inside the trapezoid $E_{M}\setminus E_{M-n}$ with vertices $(0,M-n),(0,M),(M-n,0),(M,0)$. We remark that 
$E_{M}\setminus E_{M-n}$ can be written as the disjoint union of the sets $\Pi$ and $T$ where $
\Pi $  is the parallelogram with vertices $(0,M-n),(0,M),(M-n,0),(M-n,n)$ and $T$  is the triangle with vertices $(M-n+1,0),(M-n+1,n),(M,0)$. 
That is 
\begin{align*}
\Pi&=\{
i,j: 0 \leq i < M-n+1,   M-n-i+1 \leq j \leq M-i.
\}
\\
T &= \{i,j:  M-n+1 \leq i \leq M,  0 \leq j \leq M-i\}.
\end{align*}

Thus, we compute the number of solutions to the system (\ref{eq:system}) inside $E_{M}\setminus E_{M-n}$ by computing the number of solutions inside $\Pi$ and the number of solutions inside $T$ separately:
{\small
\begin{figure}[h]
\begin{tikzpicture}

\filldraw[black] (0,0) circle (2pt) node[anchor=east] {$0$ };

\filldraw[black] (5,0) circle (2pt) node[anchor=south] {$\left(\lf\frac{M}{n}\rf-1\right)n$ };

\draw [decorate,decoration={brace,amplitude=+10pt},xshift=-0.4pt,yshift=1.3pt](5,0) -- (0,0) node[black,midway,yshift=-0.75cm] {\footnotesize $\lf\frac{M}{n}\rf-1$ solutions to $i\equiv -X$};

\draw [decorate,decoration={brace,amplitude=+10pt},xshift=0.4pt,yshift=1.3pt](10,0) -- (7,0) node[black,midway,yshift=-0.75cm] {\footnotesize one solution to $i\equiv -X$};

\draw [decorate,decoration={brace,amplitude=+10pt},xshift=0.4pt,yshift=+0.4pt](3,4) -- (3,7) node[black,midway,xshift=-55pt] {\footnotesize one solution to $j\equiv -Y$};

\filldraw[black] (7,0) circle (2pt) node[anchor=north] {$M-n+1$ };
 
\filldraw[black] (10,0) circle (2pt) node[anchor=north] {$M$ };

\draw [ ](5,3.5) -- (5,3.5) node[black,midway] { \Huge  $\Pi$};

\draw [ ](8,1) -- (8,1) node[black,midway] { \Huge  $T$};

\draw [ ](3,1.5) -- (3,1.5) node[black,midway] { \Huge  $E_{M-n}$};

\filldraw[black] (0,7) circle (2pt) node[anchor=east] {$M-n+1$ };

\filldraw[black] (0,10) circle (2pt) node[anchor=east] {$M$ };

\draw[ fill=brown, fill opacity=0.2] (7,3) -- (7,0) -- (10,0) -- cycle ;

\draw[ fill=blue, fill opacity=0.2] (0,7) -- (0,10) -- (7,3) -- (7,0) -- cycle ;

\draw (0,0) -- (10,0);
\draw (0,0) -- (0,10);
\draw (0,7) -- (7,0);
\draw (0,10) -- (10,0);
\draw (7,0) -- (7,3);

\end{tikzpicture}
\end{figure}
}
\begin{itemize}
\item If $(i,j)\in\Pi$, then $0\leq i \leq M-n
$ 
and $M-n-i
\leq j \leq M-i$. Thus, for a fixed value of $i$ satisfying $0\leq i \leq M-n$ there exist exactly $n$ values of $j$ such that $(i,j)\in\Pi$ and exactly $1$ of them satisfies the congruence $j\equiv\;-Y{\rm mod}\;n$. Hence, the number of solutions to the system (\ref{eq:system}) inside $\Pi$ is equal to the number of solutions to the congruence $i\equiv\;-X{\rm mod}\;n$ that satisfy $0\leq i \leq M-n$.

 To count the cardinality of the set 
$\left\{
i\in\mathbb{N}:0\leq i \leq M-n
\text{ and } i\equiv\;-X{\rm mod}\;n
\right\}$,
we remark that the congruence $i\equiv\;-X{\rm mod}\;n$ has exacly one solution in each interval of length $n$ and that the interval $0\leq i \leq M-n
$
 contains at least $\lf\frac{M-n
+1
 }{n}\rf=\lf\frac{M
+1
 }{n}\rf-1$ intervals of length $n$. Thus, we obtain exactly $\lf\frac{M
+1
 }{n}\rf-1$ solutions to the congruence $i\equiv\;-X{\rm mod}\;n$ in the interval $0\leq i \leq \left(\lf\frac{M
  +1
 }{n}\rf-1\right)n$ and possibly one more solution in the interval $\left(\lf\frac{M
 +1
 }{n}\rf-1\right)n< i \leq  M
 -n$. The latter has length $M
 +1
 -n\lf\frac{M
+1
 }{n}\rf$ and thus the congruence $i\equiv\;-X{\rm mod}\;n$ has a solution in that interval if and only if the remainder of the division of $X$ by $n$ is $\leq M
  +1
 -\lf\frac{M
+1
 }{n}\rf n$. Thus, the number of solutions to the system (\ref{eq:system}) inside the parallelogram $\Pi$ is given by
\[
\begin{cases}
\displaystyle\lf\frac{M
+1
}{n}\rf-1&\text{ if }\upsilon_{-X}> M
+1
-\lf\frac{M
+1
}{n}\rf n =
 v_{M+1}  \\
\displaystyle\lf\frac{M
+1
}{n}\rf &\text{ if }\upsilon_{-X}\leq M
+1
-\lf\frac{M
+1
}{n}\rf n 
=
v_{M+1}.
\end{cases}
\]

\item To count the number of solutions to the system (\ref{eq:system}) inside the triangle $T$ we observe that 
shifted horizontally to the triangle $E_{n-1}$, with vertices $(0,0),(n-1,0),(0,n-1)$, i.e $E_{n-1}$. The triangle's base has length $n$ and so the congruence $i\equiv\;-X{\rm mod}\;n$ has exactly one solution $i$.
We have that $i_0=i-(M-n+1)$ satisfies $0 \leq i_0 < n$. 
The residue $v_{i_0}$ of $i_0$ divided by $n$ equals 
\[
v_{i_0}=
\begin{cases}
v_{-X}-v_{M+1} & \text{ if } v_{-X} \geq v_{M+1}
\\
n+ v_{-X}- v_{M+1} & \text{ if } v_{-X} < v_{M+1}
\end{cases}
\] 

For this solution, the congruence $j\equiv -Y\;{\rm mod}\;n$ has one solution if and only if $\upsilon_{-X}+\upsilon_{-Y}\leq n-1$ and thus the number of solutions $n_T$ to the system (\ref{eq:system}) inside $T$ is given by
\[
n_T=
\begin{cases}
1 &\text{ if }\upsilon_{i_0}+\upsilon_{-Y}< n\\
0& \;\upsilon_{i_0}+\upsilon_{-Y}\geq n
\end{cases}
\quad \Leftrightarrow \quad 
n_T=
\begin{cases}
1 & \text{ if } v_{-X} \geq v_{M+1} \text{ and }
 v_{-X} -v_{M+1} + v_{-Y} <n 
 \\
0 & \text{ if } v_{-X} \geq v_{M+1} \text{ and }
 v_{-X} -v_{M+1} + v_{-Y} \geq n
 \\
  1 & \text{ if } v_{-X} < v_{M+1} \text{ and }
 n+v_{-X} -v_{M+1} + v_{-Y} <n 
 \\
0 & \text{ if } v_{-X} < v_{M+1} \text{ and }
 n+ v_{-X} -v_{M+1} + v_{-Y} \geq n
\end{cases}
\]
\end{itemize}

Similarly by Proposition \ref{prop:general-IJ}, we have that
\[
J(X)=J_M(X)-J_{M-n}(X)
=
n^2\cdot\#
\left\{
i\in\mathbb{N}:\lf\frac{M-n}{2}\rf+1\leq i\leq \left\lfloor\frac{M}{2}\right\rfloor\text{ and }i\equiv -X\;{\rm mod}\; n
\right\}.
\]
The computation for $J(X)$ are obvious from the definition. 
\end{proof}

We conclude this section with substituting the value of $M$ that is used to define the the $m$-polydifferentials:
\begin{mdframed}[backgroundcolor=blue!6]
\begin{corollary}\label{cor:count-IJ}
For Fermat curves of genus $g\geq 2 \Leftrightarrow n \geq 4$ and  $M=m(n-3)$ 
such that $m(n-3)+1 \geq n$
we have:
\begin{align*}
I(X,Y)&=
n^2\left(
m-\lc\frac{
 3m-1}
{n}\rc+\alpha_{X,Y}
\right),
\alpha_{X,Y}=
\begin{cases}
 -1 &\displaystyle\text{ if }\upsilon_{-X} > \upsilon_{
  1-3m}\text{ and }
 \upsilon_{-X}
 - v_{1-3m}
  +\upsilon_{-Y}\geq n\\
 1 &\displaystyle\text{ if }\upsilon_{-X}\leq\upsilon_{
  1-3m}\text{ and } 
 \upsilon_{-X}
  -v_{1-3m}+
 \upsilon_{-Y}< 0\\
 0 &\text{ otherwise }.
\end{cases} \\\\
J(X)&=
\lf
\frac{
 \lf
   \frac{ m(n-3)}{2}
 \rf	
-v_{-X}
}
{n}
\rf
-
\lc
\frac{
 \lf
   \frac{ (m-1)n-3m)}{2}
 \rf	
-v_{-X}+1
}
{n}
\rc
\end{align*}
For $M+1<n$ we have
\[
I(X,Y) =
\begin{cases}
n^2 & \text{ if } v_{-X} + v_{-Y} \leq v_{-3m}
\\
0   & \text{ otherwise}
\end{cases}
\]
and
\[
J(X) =
\lf
\frac{
 \lf
   \frac{ m(n-3)}{2}
 \rf	
-v_{-X}
}
{n}
\rf
+1
\]
In the case $m(n-3)+1<n$ we set $\alpha_{X,Y}=\frac{1}{n^2}I(X,Y) \in \{0,1\}$. 
\end{corollary}
\end{mdframed}
\begin{proof}
The formula for $J(X,Y)$ follows directly from the substitution $M=m(n-3)$. The formula for $I(X,Y)$ follows from the observation that
$
\lf\frac{m(n-3)}{n}\rf=m+\lf\frac{-3m}{n}\rf=m-\lc\frac{3m}{n}\rc.
$
For the conditions defining $I(X,Y)$, we remark that
$
\upsilon_{m(n-3)+1}=\upsilon_{1-3m}.
$
\end{proof}

\section{The Galois module structure of holomorphic poly-Differentials}\label{sec:Polydiff}

For $m\geq 1$, let $V_m=H^0(F_n,\Omega_{F_n}^{\otimes m})$ denote the $K$-vector space of global sections of holomorphic $m$-differentials on the Fermat curve $F_n$. By Theorem \ref{th:chars-m-diff}, we have that the characters of the representation $\rho_{V_m}:G\rightarrow{\rm GL}(V_m)$ are given by $\chi_{V_m}=\chi_{W_m}-\chi_{I_m}$, where $\chi_{W_m}$ and $\chi_{I_m}$ are given in Proposition \ref{prop:Wm-Im-chars}.
Notice that for $m=1$ the character $\chi_{I_1}=0$. 
Thus, the Galois module structure of $V_m$ can be computed as follows:
\begin{corollary}
Let $\chi_{\kappa,\lambda,\rho}\in{\rm Irrep}(G)$ denote any of the irreducible representations of $G$ given in Proposition \ref{prop:Irreps}. Then
\[
\langle \chi_{V_m},\chi_{\kappa,\lambda,\rho} \rangle
=\frac{1}{6n^2}
\sum_{\alpha,\beta\in\Z/n\Z\atop g\in S_3}
\left(
\chi_{W_m}(\sigma_{\alpha,\beta}g)-\chi_{I_m}(\sigma_{\alpha,\beta}g)
\right)
\overline{\chi_{\kappa,\lambda,\rho}(\sigma_{\alpha,\beta}g)}
\]
\end{corollary}

Hence, the computation of each $\langle \chi_{V_m},\chi_{\kappa,\lambda,\rho} \rangle$ breaks down in computing (at most) six sums, one for each element $g\in S_3$. We remark that by the results of the previous section, the sums corresponding to $\sigma_{\alpha,\beta}$ will be computed using the quantities $I(X,Y)$ and the sums corresponding to $\sigma_{\alpha,\beta}t,\sigma_{\alpha,\beta}ts,\sigma_{\alpha,\beta}st$ will be computed using $J(X,Y)$. For the sums corresponding to $\sigma_{\alpha,\beta}s,\sigma_{\alpha,\beta}s^2$, which appear only in the multiplicities of the irreducible representations of degree $1$ and $2$, we have the following:
\begin{lemma}\label{lem:sums-s}
For $\nu\in\{0,1,2\},\; i\in\{1,2\}$ and $\rho\in\{\rho_\mathrm{triv},\rho_\mathrm{sgn}\}$ we have that
\[
\sum_{\alpha,\beta\in\Z/n}
\chi_{V_m}(\sigma_{\alpha,\beta}s^i)
\overline{\chi_{\frac{\nu n}{3},\frac{\nu n}{3},\rho}(\sigma_{\alpha,\beta}s^i)}
=
n^2\Gamma_{\frac{\nu n}{3}}^{(m)}\in\{-n^2,0,n^2\}
\]
and for $\rho=\rho_\mathrm{stan}$ we have that
\[
\sum_{\alpha,\beta\in\Z/n}
\chi_{V_m}(\sigma_{\alpha,\beta}s^i)
\overline{\chi_{\frac{\nu n}{3},\frac{\nu n}{3},\rho_\mathrm{stan}}(\sigma_{\alpha,\beta}s^i)}
=
-n^2\Gamma_{\frac{\nu n}{3}}^{(m)}
\]
where
\[
\Gamma_{\frac{\nu n}{3}}^{(m)}:=
\begin{cases}
1&,\text{ if } 3\mid m-\nu\\
-1&,\text{ if }3\mid m-\nu+2\\
0&,\text{ otherwise }
\end{cases}
\]
\end{lemma}
\begin{proof}
By Proposition \ref{prop:Wm-Im-chars}, for $\nu\in\{0,1,2\},\; i\in\{1,2\}$ and $\rho\in\{\rho_\mathrm{triv},\rho_\mathrm{sgn}\}$, we have that if $3\mid mn$ then
\begin{align*}
\sum_{\alpha,\beta\in\Z/n}
\chi_{W_m}(\sigma_{\alpha,\beta}s^i)
\overline{\chi_{\frac{\nu n}{3},\frac{\nu n}{3},\rho}(\sigma_{\alpha,\beta}s^i)}
&=
\sum_{\alpha,\beta\in\Z/n}
\zeta^{(\alpha+\beta)\frac{mn}{3}}\overline{\zeta^{(\alpha+\beta)\frac{\nu n}{3}}}
=
\sum_{\alpha,\beta\in\Z/n}
\zeta^{(\alpha+\beta)\frac{(m-\nu)n}{3}}\\
&=
\begin{cases}
n^2&,\text{ if } n\mid \frac{(m-\nu)n}{3}\\
0&,\text{ otherwise}.
\end{cases}
=
\begin{cases}
n^2&,\text{ if } 3\mid m-\nu\\
0&,\text{ otherwise},
\end{cases}
\end{align*}
since $n\mid \frac{(m-\nu)n}{3}\Leftrightarrow\frac{(m-\nu)n}{3}=kn$ for some $k\in\mathbb{Z}\Leftrightarrow \frac{(m-\nu)}{3}=k\in\mathbb{Z}\Leftrightarrow 3\mid m-\nu$.\\

Similarly, if $3\mid (m-1)n$ then
\begin{align*}
\sum_{\alpha,\beta\in\Z/n}
\chi_{I_m}(\sigma_{\alpha,\beta}s^i)
\overline{\chi_{\frac{\nu n}{3},\frac{\nu n}{3},\rho}(\sigma_{\alpha,\beta}s^i)}
&=
\sum_{\alpha,\beta\in\Z/n}
\zeta^{(\alpha+\beta)\frac{(m-1)n}{3}}\overline{\zeta^{(\alpha+\beta)\frac{\nu n}{3}}}
=
\sum_{\alpha,\beta\in\Z/n}
\zeta^{(\alpha+\beta)\frac{(m-\nu+2)n}{3}}\\
&=
\begin{cases}
n^2&,\text{ if } n\mid \frac{(m-\nu+2)n}{3}\\
0&,\text{ otherwise}.
\end{cases}
=
\begin{cases}
n^2&,\text{ if } 3\mid m-\nu+2\\
0&,\text{ otherwise}.
\end{cases}
\end{align*}
The result follows for $\rho\in\{\rho_\mathrm{triv},\rho_\mathrm{sgn}\}$ by subtracting the two sums while for $\rho=\rho_{\mathrm{stan}}$ everything needs to be multiplied by $\chi_{\mathrm{stan}}(s)=\chi_{\mathrm{stan}}(s^2)=-1$.
\end{proof}
We proceed with applying the results of the previous section for computing the part of $\langle \chi_{V_m},\chi_{\kappa,\lambda,\rho} \rangle$ that corresponds to $\sigma_{\alpha,\beta}$:
\begin{lemma}\label{lem:sums-1}
For $\nu\in\{0,1,2\}$ and $\rho\in\{\rho_\mathrm{triv},\rho_\mathrm{sgn}, \rho_{\mathrm{stan}}\}$ we have that
\[
\sum_{\alpha,\beta\in\Z/n}
\chi_{V_m}(\sigma_{\alpha,\beta})
\overline{\chi_{\frac{\nu n}{3},\frac{\nu n}{3},\rho}(\sigma_{\alpha,\beta})}
=
\dim(\rho) \cdot 
I\left(m-\frac{\nu n}{3},m-\frac{\nu n}{3}\right)
\]
In particular for $m(n-1)-n>0$ we have 
\[
I\left(m-\frac{\nu n}{3},m-\frac{\nu n}{3}\right)=
n^2\left(
m-\lc\frac{3m
-1
}{n}\rc+A_{\frac{\nu n}{3}}^{(m)}
\right),
\]
where $A_{\frac{\nu n}{3}}^{(m)}:=\alpha_{m-\frac{\nu n}{3},m-\frac{\nu n}{3}}\in\{-1,0,1\}$ is as defined in Corollary \ref{cor:count-IJ}.
\end{lemma}
\begin{proof}
By Proposition \ref{prop:Wm-Im-chars}, for $\nu\in\{0,1,2\},\; i\in\{1,2\}$ and $\rho\in\{\rho_\mathrm{triv},\rho_\mathrm{sgn}\}$, we have that
\begin{align*}
\sum_{\alpha,\beta\in\Z/n}
\chi_{W_m}(\sigma_{\alpha,\beta})
\overline{\chi_{\frac{\nu n}{3},\frac{\nu n}{3},\rho}(\sigma_{\alpha,\beta})}
&=
\sum_{\alpha,\beta\in\Z/n}
\sum_{(i,j)\in E_{m(n-3)}}
\zeta^{\alpha(i+m)+\beta(j+m)}\overline{\zeta^{(\alpha+\beta)\frac{\nu n}{3}}}\\
&=
\sum_{\alpha,\beta\in\Z/n}
\sum_{(i,j)\in E_{m(n-3)}}
\zeta^{\alpha(i+m-\frac{\nu n}{3})+\beta(j+m-\frac{\nu n}{3})}\\
&=
I_{m(n-3)}\left(m-\frac{\nu n}{3},m-\frac{\nu n}{3}\right)
\end{align*}
and similarly
\begin{align*}
\sum_{\alpha,\beta\in\Z/n}
\chi_{I_m}(\sigma_{\alpha,\beta})
\overline{\chi_{\frac{\nu n}{3},\frac{\nu n}{3},\rho}(\sigma_{\alpha,\beta})}
=
I_{m(n-3)-n}\left(m-\frac{\nu n}{3},m-\frac{\nu n}{3}\right).
\end{align*}
By Theorem \ref{th:chars-m-diff} we obtain that
\[
\sum_{\alpha,\beta\in\Z/n}
\chi_{V_m}(\sigma_{\alpha,\beta})
\overline{\chi_{\frac{\nu n}{3},\frac{\nu n}{3},\rho}(\sigma_{\alpha,\beta})}
=I\left(m-\frac{\nu n}{3},m-\frac{\nu n}{3}\right)
\]
and the formulas follow by Corollary \ref{cor:count-IJ}. For $\rho=\rho_{\mathrm{stan}}$ we multiply the sum by $\chi_{\mathrm{stan}}(1)=\dim \rho_{\mathrm{stan}}=2$.
\end{proof}
Finally, we apply the results of the previous section for computing the parts of $\langle \chi_{V_m},\chi_{\kappa,\lambda,\rho} \rangle$ that correspond to $\sigma_{\alpha,\beta}t,\sigma_{\alpha,\beta}ts,\sigma_{\alpha,\beta}st$:
\begin{lemma}\label{lem:sums-t}
For $\nu\in\{0,1,2\}$ we have that
\begin{align*}
\sum_{g\in\{t,ts,st\}}
\sum_{\alpha,\beta\in\Z/n}
\chi_{V_m}(\sigma_{\alpha,\beta}g)
\overline{\chi_{\frac{\nu n}{3},\frac{\nu n}{3},\rho_{\mathrm{triv}}}(\sigma_{\alpha,\beta}g)}
&=
3n^2B_{\frac{\nu n}{3}}^{(m)}\in\{-3n^2,0,3n^2\}\\
\sum_{g\in\{t,ts,st\}}
\sum_{\alpha,\beta\in\Z/n}
\chi_{V_m}(\sigma_{\alpha,\beta} 
 g
)
\overline{\chi_{\frac{\nu n}{3},\frac{\nu n}{3},\rho_{\mathrm{sgn}}}(\sigma_{\alpha,\beta}
g
)}
&=
-3n^2B_{\frac{\nu n}{3}}^{(m)}\in\{-3n^2,0,3n^2\}
\end{align*}
where
\[
B_{\frac{\nu n}{3}}^{(m)}:=(-1)^m \frac{1}{n^2}J\left(m-\frac{\nu n }{3}\right)
=
\begin{cases}
1 &,\text{ if }2\mid m \text{ and } J(m-\frac{\nu n}{3})=n^2\\
-1 &,\text{ if }2\nmid m \text{ and } J(m-\frac{\nu n}{3})=n^2\\
0&,\text{ if }J(m-\frac{\nu n}{3})=0
\end{cases}
\]
and $J(m-\frac{\nu n}{3})\in\{0,1\}$ is as defined in Corollary \ref{cor:count-IJ}.
\end{lemma}
\begin{proof}
By Proposition \ref{prop:Wm-Im-chars}, for $\nu\in\{0,1,2\}$, we have that
\begin{align*}
\sum_{\alpha,\beta\in\Z/n}
\chi_{W_m}(\sigma_{\alpha,\beta}t)
\overline{\chi_{\frac{\nu n}{3},\frac{\nu n}{3},\rho_{\mathrm{triv}}}(\sigma_{\alpha,\beta}t)}
&=
(-1)^m\sum_{\alpha,\beta\in\Z/n}
\sum_{i=0}^{\lf\frac{m(n-3)}{2}\rf}
\zeta^{(\alpha-2\beta)(i+m)}\overline{\zeta^{(\alpha+\beta)\frac{\nu n}{3}}}\\
&=
(-1)^m\sum_{\alpha,\beta\in\Z/n}
\sum_{i=0}^{\lf\frac{m(n-3)}{2}\rf}
\zeta^{(\alpha-2\beta)(i+m-\frac{\nu n}{3})}\\
&=
(-1)^m J_{m(n-3)}\left(m-\frac{\nu n}{3}\right)
\end{align*}
and similarly
\begin{align*}
\sum_{\alpha,\beta\in\Z/n}
\chi_{I_m}(\sigma_{\alpha,\beta}t)
\overline{\chi_{\frac{\nu n}{3},\frac{\nu n}{3},\rho_{\mathrm{triv}}}(\sigma_{\alpha,\beta}t)}
=
(-1)^mJ_{m(n-3)-n}\left(m-\frac{\nu n}{3}\right).
\end{align*}
By Theorem \ref{th:chars-m-diff} we obtain that
\[
\sum_{\alpha,\beta\in\Z/n}
\chi_{V_m}(\sigma_{\alpha,\beta}t)
\overline{\chi_{\frac{\nu n}{3},\frac{\nu n}{3},\rho_{\mathrm{triv}}}(\sigma_{\alpha,\beta}t)}
=(-1)^mJ\left(m-\frac{\nu n}{3}\right).
\]
The same arguments give that
\[
\sum_{\alpha,\beta\in\Z/n}
\chi_{V_m}(\sigma_{\alpha,\beta}ts)
\overline{\chi_{\frac{\nu n}{3},\frac{\nu n}{3},\rho}(\sigma_{\alpha,\beta}ts)}
=
\sum_{\alpha,\beta\in\Z/n}
\chi_{V_m}(\sigma_{\alpha,\beta}st)
\overline{\chi_{\frac{\nu n}{3},\frac{\nu n}{3},\rho}(\sigma_{\alpha,\beta}st)}
=(-1)^mJ\left(m-\frac{\nu n}{3}\right)
\]
and the formulas follow by Corollary \ref{cor:count-IJ}. For $\rho=\rho_{\mathrm{sgn}}$ we multiply everything by $\chi_{\mathrm{sgn}}(t)=\chi_{\mathrm{sgn}}(ts)=\chi_{\mathrm{sgn}}(st)=-1$.
\end{proof}
We collect the above results in the following:
\begin{proposition}
For $\nu\in\{0,1,2\}$, we have that
\begin{eqnarray*}
\langle \chi_{V_m},\chi_{\frac{\nu n}{3},\frac{\nu n}{3},\rho_{\mathrm{triv}}}\rangle
 &=&
 \frac{1}{6}
 \left(
 m-\lc\frac{3m-1}{n}\rc+A_{\frac{\nu n}{3}}^{(m)}+3B_{\frac{\nu n}{3}}^{(m)}+2\Gamma_{\frac{\nu n}{3}}^{(m)}
 \right)\\
\langle \chi_{V_m},\chi_{\frac{\nu n}{3},\frac{\nu n}{3},\rho_{\mathrm{sgn}}}\rangle 
 &=&
\frac{1}{6}
\left(
 m-\lc\frac{3m-1}{n}\rc+A_{\frac{\nu n}{3}}^{(m)}-3B_{\frac{\nu n}{3}}^{(m)}+2\Gamma_{\frac{\nu n}{3}}^{(m)}
 \right)\\
\langle \chi_{V_m},\chi_{\frac{\nu n}{3},\frac{\nu n}{3},\rho_{\mathrm{stan}}}\rangle 
 &=&
\frac{1}{3}
\left(
 m-\lc\frac{3m-1}{n}\rc+A_{\frac{\nu n}{3}}^{(m)}-2\Gamma_{\frac{\nu n}{3}}^{(m)}
 \right)
\end{eqnarray*}
where each of $A_{\frac{\nu n}{3}}^{(m)},B_{\frac{\nu n}{3}}^{(m)},\Gamma_{\frac{\nu n}{3}}^{(m)}$ take values in $\{-1,0,1\}$ as in Lemmata \ref{lem:sums-s}, \ref{lem:sums-1}, \ref{lem:sums-t}.
\end{proposition}
We proceed with the multiplicities of the $3$-dimensional irreducible representations:
\begin{proposition}
For $\kappa\notin\{0,\frac{n}{3},\frac{2n}{3} \}$, we have that
\begin{eqnarray*}
\langle \chi_{V_m},\chi_{\kappa,\kappa,\rho_{\mathrm{triv}}}\rangle
 &=&
\frac{1}{2}
\left(
m-\lc\frac{3m-1}{n}\rc+A_{\kappa}^{(m)}+B_{\kappa}^{(m)}
\right)
\\
 \langle \chi_{V_m},\chi_{\kappa,\kappa,\rho_{\mathrm{sgn}}}\rangle
 &=&
\frac{1}{2}\left(
m-\lc\frac{3m-1}{n}\rc+A_{\kappa}^{(m)}-B_{\kappa}^{(m)}
\right)
\end{eqnarray*}
where
\begin{eqnarray*}
A_{\kappa}^{(m)}&=&\frac{1}{3}\left(\alpha_{m-\kappa,m-\kappa}+\alpha_{m-\kappa,m+2\kappa}+\alpha_{m+2\kappa,m-\kappa}\right)\\
B_{\kappa}^{(m)}&=&(-1)^m \frac{1}{n^2}J(m-\kappa)
\end{eqnarray*}
and $\alpha_{X,Y},\;J(m-\kappa)\in\{0,1\}$ is as defined in Corollary \ref{cor:count-IJ}.
\end{proposition}
\begin{proof}
As in the proofs of Lemmata \ref{lem:sums-1}, \ref{lem:sums-t}, Corollary \ref{cor:count-IJ} gives that
\begin{align*}
\sum_{\alpha,\beta\in\Z/n}
\chi_{V_m}(\sigma_{\alpha,\beta})
\overline{\chi_{\kappa,\kappa,\rho_\mathrm{triv}}(\sigma_{\alpha,\beta})}
&=
I\left(m-\kappa,m-\kappa\right)+I\left(m-\kappa,m+2\kappa\right)+I\left(m+2\kappa,m-\kappa\right)\\
&= n^2\left[3\left(m-\lc\frac{3m-1}{n}\rc\right)+\alpha_{m-\kappa,m-\kappa}+\alpha_{m-\kappa,m+2\kappa}+\alpha_{m+2\kappa,m-\kappa}\right]
\end{align*}
and that for $x\in\{t,ts,st\}$
\begin{align*}
\sum_{\alpha,\beta\in\Z/n}
\chi_{V_m}(\sigma_{\alpha,\beta}x)
\overline{\chi_{\kappa,\kappa,\rho_{\mathrm{triv}}}(\sigma_{\alpha,\beta}x)}
&=
(-1)^m J\left(m-\kappa\right).
\end{align*}
Thus, we obtain directly that
\begin{align*}
\langle \chi_{V_m},\chi_{\kappa,\kappa,\rho_{\mathrm{triv}}}\rangle
&=
\frac{1}{6n^2}
\sum_{\alpha,\beta\in\Z/n \atop g\in S_3}
\chi_{V_m}(\sigma_{\alpha,\beta})
\overline{\chi_{\kappa,\kappa,\rho}(\sigma_{\alpha,\beta})}\\
&=
\frac{1}{6}
\left(
3\left(m-\lc\frac{3m-1}{n}\rc\right)+\alpha_{m-\kappa,m-\kappa}+\alpha_{m-\kappa,m+2\kappa}+\alpha_{m+2\kappa,m-\kappa}
+3(-1)^m \frac{1}{n^2}J\left(m-\kappa\right)
\right)\\
&=
\frac{1}{2}
\left(
m-\lc\frac{3m-1}{n}\rc+\frac{1}{3}\left(\alpha_{m-\kappa,m-\kappa}+\alpha_{m-\kappa,m+2\kappa}+\alpha_{m+2\kappa,m-\kappa}\right)
+(-1)^m \frac{1}{n^2}J\left(m-\kappa\right)
\right)
\end{align*}
Substituting 
\begin{eqnarray*}
A_{\kappa}&=&\frac{1}{3}\left(\alpha_{m-\kappa,m-\kappa}+\alpha_{m-\kappa,m+2\kappa}+\alpha_{m+2\kappa,m-\kappa}\right)\\
B_{\kappa}&=&(-1)^m \frac{1}{n^2}J(m-\kappa)
\end{eqnarray*}
gives the desired formula for $\rho_{\mathrm{triv}}$. The result for $\rho_{\mathrm{sgn}}$ follows in the same manner.
\end{proof}
Finally, we obtain the multiplicities of the $6$-dimensional representations:
\begin{proposition}
For $\kappa\notin\{0,\frac{n}{3},\frac{2n}{3} \}$, we have that
\begin{eqnarray*}
\langle \chi_{V_m},\chi_{\kappa,\kappa,\rho_{\mathrm{triv}}}\rangle
 &=&
m-\lc\frac{3m-1}{n}\rc+A_{\kappa,\lambda}^{(m)}
\end{eqnarray*}
where
\begin{eqnarray*}
A_{\kappa,\lambda}^{(m)}&=&\frac{1}{6}
\left(
\alpha_{m-\kappa,m-\lambda}
+\alpha_{m-\lambda,m-\kappa}
+\alpha_{m-\kappa,m+\kappa+\lambda}
+\alpha_{m+\kappa+\lambda,m-\kappa}
+\alpha_{m-\lambda,m+\kappa+\lambda}
+\alpha_{m+\kappa+\lambda,m-\lambda}
\right)
\end{eqnarray*}
and $\alpha_{X,Y}\in\{0,1\}$ is as defined in Corollary \ref{cor:count-IJ}.
\end{proposition}
\begin{proof}
The result follows by observing that:
\begin{align*}
\sum_{\alpha,\beta\in\Z/n}
\chi_{V_m}(\sigma_{\alpha,\beta})
\overline{\chi_{\kappa,\lambda,\rho_\mathrm{triv}}(\sigma_{\alpha,\beta})}
=&\;
I\left(m-\kappa,m-\lambda\right)+I\left(m-\lambda,m-\kappa\right)+I\left(m-\kappa,m+\kappa+\lambda\right)\\
&+I\left(m+\kappa+\lambda,m-\kappa\right)+I\left(m-\lambda,m+\kappa+\lambda\right)+I\left(m+\kappa+\lambda,m-\lambda\right)\\
=&\; n^2\bigg[6\left(m-\lc\frac{3m-1}{n}\rc\right)
+\alpha_{m-\kappa,m-\lambda}
+\alpha_{m-\lambda,m-\kappa}
+\alpha_{m-\kappa,m+\kappa+\lambda}\\
&+\alpha_{m+\kappa+\lambda,m-\kappa}
+\alpha_{m-\lambda,m+\kappa+\lambda}
+\alpha_{m+\kappa+\lambda,m-\lambda}
\bigg]
\end{align*}
\end{proof}
We conclude this section by collecting the above results in the following:\\
\begin{mdframed}[backgroundcolor=blue!6]
\begin{theorem}\label{th:decompose-all}
\[
\langle \chi_{V_m},\chi_{\kappa,\lambda,\rho}\rangle =
\frac{\dim \rho}{6}
\left(
m-
\lc\frac{3m-1}{n}\rc
 +A_{\kappa,\lambda,\rho}^{(m)}
 \right)
+\frac{1}{2}B_{\kappa,\lambda,\rho}^{(m)}+\frac{1}{3}\Gamma_{\kappa,\lambda,\rho}^{(m)}
\]
where $A_{\kappa,\lambda,\rho}^{(m)},B_{\kappa,\lambda,\rho}^{(m)}$ are defined using the quantities $\alpha_{X,Y}$ and $J(X)$ of Corollary \ref{cor:count-IJ} as follows:
\[
{\renewcommand{\arraystretch}{1.5}
\begin{array}{|c|c|c|}
\hline
\kappa,\lambda,\rho & A_{\kappa,\lambda,\rho}^{(m)} & B_{\kappa,\lambda,\rho}^{(m)}  \\
\hline
\frac{\nu n}{3},\frac{\nu n}{3},\rho_\mathrm{triv} & \alpha_{m-\frac{\nu n}{3},m-\frac{\nu n}{3}}
&\frac{(-1)^m}{n^2} J\left(m-\frac{\nu n}{3}\right)
 \\
\hline
\frac{\nu n}{3},\frac{\nu n}{3},\rho_\mathrm{sgn} & \alpha_{m-\frac{\nu n}{3},m-\frac{\nu n}{3}}
& \frac{(-1)^{m+1}}{n^2} J\left(m-\frac{\nu n}{3}\right)
\\
\hline
\frac{\nu n}{3},\frac{\nu n}{3},\rho_\mathrm{stan} & \alpha_{m-\frac{\nu n}{3},m-\frac{\nu n}{3}} 
& 0
\\
\hline
\kappa,\kappa,\rho_\mathrm{triv} & \frac{1}{3}\left(\alpha_{m-\kappa,m-\kappa}+\alpha_{m-\kappa,m+2\kappa}+\alpha_{m+2\kappa,m-\kappa}\right) 
& \frac{(-1)^m}{n^2} J\left(m-\kappa\right)
\\
\hline
\kappa,\kappa,\rho_\mathrm{sgn} & \frac{1}{3}\left(\alpha_{m-\kappa,m-\kappa}+\alpha_{m-\kappa,m+2\kappa}+\alpha_{m+2\kappa,m-\kappa}\right) 
& \frac{(-1)^{m+1}}{n^2} J\left(m-\kappa\right) 
\\
\hline
\kappa,\lambda,\rho_\mathrm{triv} & \frac{1}{6}
\bigg(
\alpha_{m-\kappa,m-\lambda}
+\alpha_{m-\lambda,m+\kappa+\lambda}
+\alpha_{m-\kappa,m+\kappa+\lambda} &
 \\
& 
+\alpha_{m+\kappa+\lambda,m-\kappa}
+\alpha_{m-\lambda,m-\kappa}
+\alpha_{m+\kappa+\lambda,m-\lambda}
\bigg) &
0  
\\
\hline
\end{array}
}
\]
and
\begin{eqnarray*}
 \Gamma_{\frac{\nu n}{3},\frac{\nu n}{3},\rho_\mathrm{triv}}^{(m)}=
\Gamma_{\frac{\nu n}{3},\frac{\nu n}{3},\rho_\mathrm{sgn}}^{(m)}=
-\Gamma_{\frac{\nu n}{3},\frac{\nu n}{3},\rho_\mathrm{stan}}^{(m)}&=&
\begin{cases}
1&,\text{ if } 3\mid m-\nu\\
-1&,\text{ if }3\mid m-\nu+2 \text{ and } m(n-3) \geq n\\
0&,\text{ otherwise }
\end{cases}\\\\
 \Gamma_{\kappa,\kappa,\rho_\mathrm{triv}}^{(m)}=
  \Gamma_{\kappa,\kappa,\rho_\mathrm{sgn}}^{(m)}=
   \Gamma_{\kappa,\lambda,\rho_\mathrm{triv}}^{(m)}&=&
   0
\end{eqnarray*}
\end{theorem}
\end{mdframed}
\bigskip
To convince the reader that the above information can give explicit results, we have included Table \ref{tab:dimension-comps}, which treats Fermat curves corresponding to $n=4,5,6$. The first $9$
columns contain the multiplicities of one-dimensional and two-dimensional representations. For the $3$ dimensional representations $\theta_{\kappa,\kappa,\rho}$, $\rho\in \{\rho_{\mathrm{triv}}, \rho_{\mathrm{sgn}}\}$ if the symbol $[\kappa,t]$ appears then this means that the representation 
$\theta_{\kappa,\kappa,\rho}$ appears with multiplicity $t$. Similarly if $[(\kappa,\lambda), t]$ appears then 
$\theta_{\kappa,\lambda,\rho_{\mathrm{triv}}}$ appears with multiplicity $t$. For example the $6$th row indicates that for the curve $F_4:x^4+y^4+z^4=0$ we have the decomposition
\[
H^0(F_4,\Omega_{F_4}^{\otimes 6 })=\theta_{0,0,\mathrm{triv}}
\oplus 
\theta_{1,1,\rho_{\mathrm{triv}}} 
\oplus
\theta_{2,2,\rho_{\mathrm{triv}}}
\oplus
\theta_{3,3,\rho_{\mathrm{triv}}}
\oplus
\theta_{2,2,\rho_{\mathrm{sgn}}}
\oplus
\theta_{3,3,\rho_{\mathrm{sgn}}}
\oplus
\theta_{0,1,\rho_{\mathrm{triv}}}.
\]
We remark that the last column of Table \ref{tab:dimension-comps} serves as an extra verification that the multiplicities add to the expected $k$-dimension of $H^0(F_n,\Omega_{F_n}^{\otimes m })$.

\begin{landscape}
\begin{table}
\caption{ \label{tab:dimension-comps} The $k[G]$-module structure of $H^0(X,\Omega_X^{\otimes m})$ for $n\in\{4,5,6\}$ and $m\in\{1,2\ldots,9\}$.}
\[
\begin{array}{|c|c| c|c|c |c|c|c |c|c|c  |c|c|c| c|}
\hline
n & m & \multicolumn{3}{c}{\rho_\mathrm{triv}} 
&
\multicolumn{3}{|c|}{\rho_\mathrm{sgn}}
&
\multicolumn{3}{|c|}{\rho_\mathrm{stan}}
&
\kappa,\kappa, \rho_\mathrm{triv}
&
\kappa,\kappa, \rho_\mathrm{sgn}
&
\kappa,\lambda 
& \dim H^{0}(X,\Omega_X^{\otimes m})
\\
\hline
4 & 1 & 0 & - & - & 0  &- &- &0   &- & -&  &   [1,1] &   & 3  \\
4 & 2 & 0 & - & - & 0  &- &- &0   &- & -&   [2,1], [3,1] &  &   & 6  \\
4 & 3 & 0 & - & - & 1  &- &- &0   &- & -&  &   [3,1] & [(0,1),1]  & 10  \\
4 & 4 & 0 & - & - & 0  &- &- &1   &- & -&   [1,1], [2,1] &  & [(0,1),1]  & 14  \\
4 & 5 & 0 & - & - & 0  &- &- &0   &- & -&   [1,1] &   [1,1], [2,1], [3,1] & [(0,1),1]  & 18  \\
4 & 6 & 1 & - & - & 0  &- &- &0   &- & -&   [1,1], [2,1], [3,1] &   [2,1], [3,1] & [(0,1),1]  & 22  \\
4 & 7 & 0 & - & - & 0  &- &- &1   &- & -&   [3,1] &   [1,1], [2,1], [3,1] & [(0,1),2]  & 26  \\
4 & 8 & 1 & - & - & 0  &- &- &1   &- & -&   [1,1], [2,1], [3,1] &   [1,1], [2,1] & [(0,1),2]  & 30  \\
4 & 9 & 0 & - & - & 1  &- &- &0   &- & -&   [1,1], [2,1], [3,1] &   [1,2], [2,1], [3,1] & [(0,1),2]  & 34  \\
4 & 10 & 0 & - & - & 0  &- &- &1   &- & -&   [1,1], [2,2], [3,2] &   [1,1], [2,1], [3,1] & [(0,1),2]  & 38  \\
4 & 11 & 0 & - & - & 1  &- &- &1   &- & -&   [1,1], [2,1], [3,1] &   [1,1], [2,1], [3,2] & [(0,1),3]  & 42  \\
4 & 12 & 1 & - & - & 1  &- &- &1   &- & -&   [1,2], [2,2], [3,1] &   [1,1], [2,1], [3,1] & [(0,1),3]  & 46  \\
4 & 13 & 0 & - & - & 0  &- &- &1   &- & -&   [1,2], [2,1], [3,1] &   [1,2], [2,2], [3,2] & [(0,1),3]  & 50  \\
5 & 1 & 0 & - & - & 0  &- &- &0   &- & -&  &   [1,1], [2,1] &   & 6  \\
5 & 2 & 0 & - & - & 0  &- &- &0   &- & -&   [2,1], [3,1], [4,1] &  & [(0,2),1]  & 15  \\
5 & 3 & 0 & - & - & 1  &- &- &0   &- & -&   [3,1] &   [1,1], [3,1], [4,1] & [(0,1),1], [(0,2),1]  & 25  \\
5 & 4 & 0 & - & - & 0  &- &- &1   &- & -&   [1,1], [2,1], [3,1], [4,1] &   [4,1] & [(0,1),2], [(0,2),1]  & 35  \\
5 & 5 & 0 & - & - & 1  &- &- &1   &- & -&   [1,1], [2,1] &   [1,1], [2,1], [3,1], [4,1] & [(0,1),2], [(0,2),2]  & 45  \\
5 & 6 & 1 & - & - & 0  &- &- &0   &- & -&   [1,2], [2,2], [3,1], [4,1] &   [1,1], [2,1], [3,1], [4,1] & [(0,1),2], [(0,2),2]  & 55  \\
5 & 7 & 0 & - & - & 0  &- &- &1   &- & -&   [1,1], [2,1], [3,1], [4,1] &   [1,1], [2,2], [3,2], [4,2] & [(0,1),2], [(0,2),3]  & 65  \\
5 & 8 & 1 & - & - & 0  &- &- &1   &- & -&   [1,2], [2,1], [3,2], [4,2] &   [1,1], [2,1], [3,2], [4,1] & [(0,1),3], [(0,2),3]  & 75  \\
5 & 9 & 1 & - & - & 1  &- &- &1   &- & -&   [1,1], [2,1], [3,1], [4,2] &   [1,2], [2,2], [3,2], [4,2] & [(0,1),4], [(0,2),3]  & 85  \\
5 & 10 & 1 & - & - & 0  &- &- &2   &- & -&   [1,2], [2,2], [3,2], [4,2] &   [1,2], [2,2], [3,1], [4,1] & [(0,1),4], [(0,2),4]  & 95  \\
5 & 11 & 0 & - & - & 1  &- &- &1   &- & -&   [1,2], [2,2], [3,2], [4,2] &   [1,3], [2,3], [3,2], [4,2] & [(0,1),4], [(0,2),4]  & 105  \\
5 & 12 & 1 & - & - & 1  &- &- &1   &- & -&   [1,2], [2,3], [3,3], [4,3] &   [1,2], [2,2], [3,2], [4,2] & [(0,1),4], [(0,2),5]  & 115  \\
5 & 13 & 0 & - & - & 1  &- &- &2   &- & -&   [1,2], [2,2], [3,3], [4,2] &   [1,3], [2,2], [3,3], [4,3] & [(0,1),5], [(0,2),5]  & 125  \\
6 & 1 & 0 & 0 & 0 & 0 & 1 & 0 & 0 & 0 & 0 & &   [1,1] & [(1,2),1]  & 10  \\
6 & 2 & 0 & 0 & 0 & 0 & 0 & 1 & 1 & 0 & 0 &  [3,1], [5,1] &  & [(0,2),1], [(1,2),1], [(3,4),1]  & 27  \\
6 & 3 & 0 & 1 & 0 & 0 & 0 & 0 & 0 & 0 & 1 &  [3,1] &   [1,1], [3,1], [5,1] & [(0,1),1], [(0,2),1], [(1,2),1], [(3,4),2]  & 45  \\
6 & 4 & 0 & 0 & 1 & 1 & 0 & 0 & 1 & 0 & 1 &  [1,1], [3,1], [5,1] &   [1,1], [5,1] & [(0,1),2], [(0,2),2], [(1,2),1], [(3,4),2]  & 63  \\
6 & 5 & 0 & 1 & 1 & 0 & 0 & 1 & 0 & 1 & 0 &  [1,1], [3,1], [5,1] &   [1,1], [3,1], [5,2] & [(0,1),3], [(0,2),2], [(1,2),2], [(3,4),2]  & 81  \\
6 & 6 & 1 & 1 & 1 & 1 & 0 & 1 & 0 & 0 & 1 &  [1,2], [3,2], [5,1] &   [1,1], [3,1], [5,1] & [(0,1),3], [(0,2),3], [(1,2),3], [(3,4),2]  & 99  \\
6 & 7 & 0 & 0 & 1 & 1 & 1 & 1 & 0 & 1 & 1 &  [1,2], [3,1], [5,1] &   [1,2], [3,2], [5,2] & [(0,1),3], [(0,2),3], [(1,2),4], [(3,4),3]  & 117  \\
6 & 8 & 1 & 0 & 1 & 1 & 0 & 2 & 1 & 1 & 1 &  [1,2], [3,2], [5,2] &   [1,1], [3,2], [5,2] & [(0,1),3], [(0,2),4], [(1,2),4], [(3,4),4]  & 135  \\
6 & 9 & 1 & 1 & 1 & 0 & 1 & 1 & 0 & 1 & 2 &  [1,2], [3,2], [5,2] &   [1,2], [3,3], [5,2] & [(0,1),4], [(0,2),4], [(1,2),4], [(3,4),5]  & 153  \\
\hline
\end{array}
\]
\end{table}
\end{landscape}

\section{Generating Functions}\label{sec:GenFunctions}
Let
$
R=\bigoplus_{m=0}^\infty H^0(F_n,\Omega_{F_n}^{\otimes m})
$
denote the canonical ring of the Fermat curve $F_n:x^n+y^n+z^n=1$ and let $G=\left(\Z/n\Z\times\Z/n\Z\right)\rtimes S_3$ be its automorphism group. The equivariant Hilbert function of the action of $G$ on $R$ is defined as
\[
H_{R,G}(t)=\sum_{m=0}^\infty[H^0(F_n,\Omega_{F_n}^{\otimes m})]t^m
\]
where $[H^0(F_n,\Omega_{F_n}^{\otimes m})]$ denotes the class in the Grothendieck group $K_0(G,k)$. For each irreducible representation $\theta_{\kappa,\lambda,\rho}$ of $G$ we denote by $H_{\kappa,\lambda,\rho}(t)$ the equivariant Hilbert series of the respective isotypical component of the action of $G$ on $R$, so that
\[
H_{R,G}(t)=\sum_{\kappa,\lambda,\rho}H_{\kappa,\lambda,\rho}(t).
\]
We then have the following:\\

\begin{mdframed}[backgroundcolor=blue!6]
\begin{theorem}\label{th:Hilbert}
Let $\theta_{\kappa,\lambda,\rho}\in{\rm Irrep}(G)$ be as in Proposition \ref{prop:Irreps}, let $A_{\kappa,\lambda,\rho}^{(m)}$ and $B_{\kappa,\lambda,\rho}^{(m)}$ be as in Theorem \ref{th:decompose-all} and let
\[
n_{0,\kappa,\lambda,\rho}:=
\begin{cases}
1&,\text{ if }\theta_{\kappa,\lambda,\rho}=\theta_{0,0,\rho_\mathrm{triv}}\\
0&,\text{ otherwise. }
\end{cases}
\]
The equivariant Hilbert function $H_{\kappa,\lambda,\rho}(t)$ is given by the following rational function
\begin{multline*}
H_{\kappa,\lambda,\rho}(t)
=
\left(
n_{0,\kappa,\lambda,\rho}-
\frac{\dim\rho}{6}A_{\kappa,\lambda,\rho}^{(n)}
-\frac{1}{2}B_{\kappa,\lambda,\rho}^{(2n)}
\right)
\\
+\frac{\dim\rho}{6}
\left[
\frac{t}{(1-t)^2}
-\frac{1}{1-t^n}
\left(
\frac{3t^n}{1-t}
+F(t)+G_{A_{\kappa,\lambda,\rho}}(t)
\right)
\right]
+\frac{1}{1-t^{2n}}G_{B_{\kappa,\lambda,\rho}}(t)
+\frac{1}{3}G_{\Gamma_{\kappa,\lambda,\rho}}(t)
\end{multline*}
where
\begin{eqnarray*}
F(t)&=& \sum_{m=0}^{n-1}\lc  \frac{3m-1}{n}\rc t^m
\\
G_{A_{\kappa,\lambda,\rho}}(t)&=& 
 \sum_{m=0}^{n-1} A_{\kappa,\lambda,\rho}^{(m+n)} t^{m}
\\
G_{B_{\kappa,\lambda,\rho}}(t)&=& 
 \sum_{m=0}^{2n-1} B_{\kappa,\lambda,\rho}^{(m+2n)} t^{m}
\\
G_{\Gamma_{\kappa,\lambda,\rho_\mathrm{triv}}}(t)=G_{\Gamma_{\kappa,\lambda,\rho_\mathrm{sgn}}}(t)=-G_{\Gamma_{\kappa,\lambda,\rho_\mathrm{stan}}}(t)&=& 
\begin{cases}
\displaystyle\frac{t^\nu-t^{\nu +1 }}{1-t^3}&,\text{ if }\kappa=\lambda=\frac{\nu n}{3}\text{ and }\nu\in\{0,1,2\}\\
0&,\text{ otherwise}.
\end{cases}
\\
\end{eqnarray*}
\end{theorem}
\end{mdframed}
\begin{proof}
For $m\geq 0$, let $V_m=H^0(F_n,\Omega_{F_n}^{\otimes m})$. Observe that $V_0=k$ and so
$
\langle \chi_{V_0},\chi_{\kappa,\lambda,\rho}\rangle =n_{0,\kappa,\lambda,\rho}
$,
whereas for $m\geq 1, \;\langle \chi_{V_m},\chi_{\kappa,\lambda,\rho}\rangle$ is given by Theorem \ref{th:decompose-all}. Thus
\begin{eqnarray*}
H_{\kappa,\lambda,\rho}(t)
&=&
\sum_{m=0}^{\infty} \langle \chi_{V_m},\chi_{\kappa,\lambda,\rho} \rangle\\
&=&
n_{0,\kappa,\lambda,\rho}
+
\sum_{m=1}^{\infty} \left[
\frac{\dim \rho}{6}
\left(
m-
\lc\frac{3m-1}{n}\rc
 +A_{\kappa,\lambda,\rho}^{(m)}
\right)
+\frac{1}{2}B_{\kappa,\lambda,\rho}^{(m)}+\frac{1}{3}\Gamma_{\kappa,\lambda,\rho}^{(m)}
\right]
 t^m.
\end{eqnarray*}
Note that the quantities $A_{\kappa,\lambda,\rho}^{(m)}$ and $B_{\kappa,\lambda,\rho}^{(m)}$ are defined only for $m\geq 1$ and have no meaning for $m=0$. However, we observe that by Theorem \ref{th:decompose-all} and Corollary \ref{cor:count-IJ}, the quantity $A_{\kappa,\lambda,\rho}^{(m)}$ depends only on the value of $m$ modulo $n$ and thus we extend the definition by setting $A_{\kappa,\lambda,\rho}^{(0)}:=A_{\kappa,\lambda,\rho}^{(n)}$. Further, recall that for $n$ large enough and for any $\kappa\in\Z/n\Z$, by Corollary \ref{cor:count-IJ}
\begin{align*}
J(m-\kappa)
  &=
  \lf
  \frac{
   \lf
     \frac{ m(n-3)}{2}
   \rf	
  -v_{-m+\kappa}
  }
  {n}
  \rf
  -
  \lc
  \frac{
   \lf
     \frac{ m(n-3)-2)}{2}
   \rf	
  -v_{-m+\kappa}+1
  }
  {n}
\rc.
\end{align*}
Thus, if $m\equiv m'\;{\rm mod}\; 2n$ then $J(m-\kappa)=J(m'-\kappa)$ and, by Theorem \ref{th:decompose-all}, $B_{\kappa,\lambda,\rho}^{(m)}=B_{\kappa,\lambda,\rho}^{(m')}$. This allows us to extend the definition by setting $B_{\kappa,\lambda,\rho}^{(0)}:=B_{\kappa,\lambda,\rho}^{(2n)}$; finally, we observe that by Theorem \ref{th:decompose-all} $\Gamma_{\kappa,\lambda,\rho}^{(0)}=0$ and thus we set
\begin{eqnarray*}
H'_{\kappa,\lambda,\rho}(t)
&=&
\sum_{m=0}^{\infty} \left[
\frac{\dim \rho}{6}
\left(
m-
\lc\frac{3m-1}{n}\rc
 +A_{\kappa,\lambda,\rho}^{(m)}
\right)
+\frac{1}{2}B_{\kappa,\lambda,\rho}^{(m)}+\frac{1}{3}\Gamma_{\kappa,\lambda,\rho}^{(m)}
\right]
 t^m
\\
&=&
\left(
\frac{\dim\rho}{6}A_{\kappa,\lambda,\rho}^{(n)}
+\frac{1}{2}B_{\kappa,\lambda,\rho}^{(2n)}
\right)
+\sum_{m=1}^{\infty} \left[
\frac{\dim \rho}{6}
\left(
m-
\lc\frac{3m-1}{n}\rc
 +A_{\kappa,\lambda,\rho}^{(m)}
\right)
+\frac{1}{2}B_{\kappa,\lambda,\rho}^{(m)}+\frac{1}{3}\Gamma_{\kappa,\lambda,\rho}^{(m)}
\right]
 t^m.
\end{eqnarray*}
To compute $H'_{\kappa,\lambda,\rho}(t)$, we rewrite
\begin{eqnarray*}
H'_{\kappa,\lambda,\rho}(t)
=
\frac{\dim \rho}{6}
\sum_{m=0}^\infty
\left(
m-
\lc\frac{3m-1}{n}\rc
 +A_{\kappa,\lambda,\rho}^{(m)}
 \right)t^m
+
\frac{1}{2}\sum_{m=0}^\infty B_{\kappa,\lambda,\rho}^{(m)}t^m
+
\frac{1}{3}\sum_{m=0}^\infty\Gamma_{\kappa,\lambda,\rho}^{(m)}t^m,
\end{eqnarray*}
then observe that
\[
\sum_{m=0}^\infty
m t^m
=
\frac{t}{(1-t)^2},
\]
and similarly
\begin{eqnarray*}
\sum_{m=0}^\infty \lc  \frac{3m-1}{n}\rc  t^m 
&=&
\sum_{\upsilon=0}^{n-1}\sum_{\mu=0}^\infty \lc  \frac{3\mu n+3\upsilon-1}{n}\rc  t^{\mu n+\upsilon} 
=
\sum_{\upsilon=0}^{n-1}
 t^\upsilon
\sum_{\mu=0}^\infty 
 3\mu  t^{\mu n}
+
\sum_{\upsilon=0}^{n-1}
\lc  \frac{3\upsilon-1}{n}\rc t^\upsilon
\sum_{\mu=0}^\infty t^{\mu n} \\
&=&
\frac{3t^n }{(1-t^n)^2} 
\sum_{v=0}^{n-1} t^v  
+
\frac{1}{1-t^n} \sum_{\upsilon=0}^{n-1}\lc  \frac{3\upsilon-1}{n}\rc t^\upsilon 
=
\frac{1}{1-t^n}
\left( 
\frac{3t^n}{1-t} + \sum_{\upsilon=0}^{n-1}\lc  \frac{3\upsilon-1}{n}\rc t^\upsilon 
\right).
\end{eqnarray*}
Next, we recall that by Theorem \ref{th:decompose-all}, $\Gamma_{\kappa,\lambda,\rho}$ is non-zero only if $\kappa=\lambda=\frac{\nu n}{3}$, in which case
\[
\sum_{m=0}^\infty\Gamma_{\frac{\nu n}{3},\frac{\nu n}{3}}^{(m)} t^m=
\sum_{\mu=0}^\infty t^{3\mu+\nu}-\sum_{\mu=0}^\infty
 t^{3\mu+\nu  +1}
=
\frac{t^\nu-t^{\nu +1 }}{1-t^3}
\]
\[\]
The arguments used above that $A_{\kappa,\lambda,\rho}^{(m)}$ depends only on the value of $m$ modulo $n$ and $B_{\kappa,\lambda,\rho}^{(m)}$ depends only on the value of $m$ modulo $2n$ give that:
\begin{eqnarray*}
\sum_{m=0}^{\infty} A_{\kappa,\lambda,\rho}^{(m)} t^m&=&
\sum_{\mu=0}^{\infty} \sum_{m =0}^{n-1} A_{\kappa,\lambda,\rho}^{(m)} t^{\mu n+m}
=
\frac{1}{1-t^n} \sum_{m=0}^{n-1} A_{\kappa,\lambda,\rho}^{(m)} t^{m}\\
\sum_{m=0}^{\infty} B_{\kappa,\lambda,\rho}^{(m)} t^m&=&
\sum_{\mu=0}^{\infty} \sum_{m =0}^{2n-1} B_{\kappa,\lambda,\rho}^{(m)} t^{\mu\cdot 2 n+m}
=
\frac{1}{1-t^{2n}} \sum_{m=0}^{2n-1} B_{\kappa,\lambda,\rho}^{(m)} t^{m}
\end{eqnarray*}
Finally, we replace $A_{\kappa,\lambda,\rho}^{(m)}$ and $B_{\kappa,\lambda,\rho}^{(m)}$ by $A_{\kappa,\lambda,\rho}^{(m+n)} $ and $B_{\kappa,\lambda,\rho}^{(m+2n)}$ respectively, since these quantities have no meaning for $m=0$. 
\end{proof}
Although the expression for $H_{\kappa,\lambda,\rho}(t)$ given in Theorem \ref{th:Hilbert} is complicated, it serves well for computations as it has allowed us to write a program that computes the equivariant Hilbert function of any Fermat curve; see the example below:

\begin{example}
We consider the Fermat curve $F_6:x^6+y^6+z^6=1$, which has genus $g=10$ and automorphism group $G=\left(\Z/6\Z\times\Z/6\Z\right)\rtimes S_3$. By Proposition \ref{prop:Irreps}, $G$ has the following 19 irreducible representations:
\[
\begin{array}{|c|c|}
\hline
&
\theta_{0,0,\rho_\mathrm{triv}},\;\theta_{2,2,\rho_\mathrm{triv}} ,\;\theta_{4,4,\rho_\mathrm{triv}} 
\\
\text{Dimension 1} &
\theta_{0,0,\rho_\mathrm{sgn}},\;\theta_{2,2,\rho_\mathrm{sgn}} ,\;\theta_{4,4,\rho_\mathrm{sgn}} 
\\
\hline
\text{Dimension 2} 
&
\theta_{0,0,\rho_\mathrm{stan}},\;\theta_{2,2,\rho_\mathrm{stan}} ,\;\theta_{4,4,\rho_\mathrm{stan}} 
\\
\hline
&
\theta_{1,1,\rho_\mathrm{triv}},\;\theta_{3,3,\rho_\mathrm{triv}} ,\;\theta_{5,5,\rho_\mathrm{triv}} 
\\
\text{Dimension 3} &
\theta_{1,1,\rho_\mathrm{sgn}},\;\theta_{3,3,\rho_\mathrm{sgn}} ,\;\theta_{5,5,\rho_\mathrm{sgn}} 
\\
\hline
\text{Dimension 6} 
&
\theta_{0,1,\rho_\mathrm{triv}},\;\theta_{0,2,\rho_\mathrm{triv}} ,\;\theta_{1,2,\rho_\mathrm{triv}}, \;\theta_{3,4,\rho_\mathrm{triv}}
\\
\hline
\end{array}
\]
We use our computer code\footnote{File: FinalCodeFermatReps.ipynb, url: \url{shorturl.at/fnzIR}} to explictly compute the equivariant Hilbert functions $H_{\kappa,\lambda,\rho}(t)$ of the respective isotypical components, which are given in the following table:
{\renewcommand{\arraystretch}{2.5}
\[
\begin{array}{|c|c|c|c|}
\hline
(\kappa,\lambda)
&\rho_\mathrm{triv}
&\rho_\mathrm{sgn}
&\rho_\mathrm{stan}\\
\hline
(0,0)
&\displaystyle\frac{ t^{14} -  t^{13} +  t^{8} -  t^{7} +  t^{6} -  t + 1}{ t^{13} -  t^{12} -  t + 1}
&\displaystyle\frac{ t^{11} -  t^{10} +  t^{9} -  t^{7} +  t^{5} -  t^{4} +  t^{3}}{ t^{13} -  t^{12} -  t + 1}
&\displaystyle\frac{t^{4}}{t^{7} - t^{6} - t + 1}\\
\hline
(2,2)
&\displaystyle\frac{ t^{12} -  t^{11} +  t^{10} -  t^{9} +  t^{6} -  t^{5} +  t^{4}}{ t^{13} -  t^{12} -  t + 1}
&\displaystyle\frac{ t^{9} -  t^{8} +  t^{7} -  t^{2} +  t}{ t^{13} -  t^{12} -  t + 1}
&\displaystyle\frac{t^{5} - t^{3} + t^{2}}{t^{7} - t^{6} -t + 1}\\
\hline
(4,4)
&\displaystyle\frac{ t^{10} -  t^{9} +  t^{8} -  t^{5} +  t^{4} -  t^{3} +  t^{2}}{ t^{13} -  t^{12} -  t + 1}
&\displaystyle\frac{ t^{13} -  t^{12} +  t^{7} -  t^{6} +  t^{5}}{ t^{13} -  t^{12} -  t + 1}
&\displaystyle\frac{t^{6} - t^{5} + t^{3}}{t^{7} -t^{6} - t + 1}\\
\hline 
(1,1)
&\displaystyle\frac{ t^{12} -  t^{11} +  t^{10} +  t^{6} +  t^{4}}{ t^{13} -  t^{12} -  t + 1}
& \displaystyle \frac{ t^{11} +  t^{9} -  t^{8} +  t^{7} +  t^{3} -  t^{2} +  t}{ t^{13} -  t^{12} -  t + 1}
& -
\\
\hline
(3,3)
&\displaystyle\frac{ t^{12} +  t^{8} -  t^{7} +  t^{6} +  t^{2}}{ t^{13} -  t^{12} -  t + 1}
&\displaystyle\frac{ t^{11} -  t^{10} +  t^{9} +  t^{7} +  t^{5} -  t^{4} +  t^{3}}{ t^{13} -  t^{12} -  t + 1}
& -
\\
\hline
(5,5)
&\displaystyle \frac{ t^{10} +  t^{8} +  t^{4} -  t^{3} +  t^{2}}{ t^{13} -  t^{12} -  t + 1}
& \displaystyle \frac{ t^{13} -  t^{12} +  t^{11} +  t^{7} -  t^{6} +  t^{5} +  t^{3}}{ t^{13} -  t^{12} -  t + 1}
& -
\\
\hline
(0,1)
&\displaystyle \frac{ t^{3}}{ t^{5} - 2  t^{4} +  t^{3} +  t^{2} - 2  t + 1}
&  -
& -
\\
\hline
(0,2)
& \displaystyle \frac{ t^{2}}{ t^{3} -  t^{2} -  t + 1}
& -
& -
\\
\hline
(1,2)
&\displaystyle\frac{ t^{4} -  t^{2} +  t}{ t^{5} - 2  t^{4} +  t^{3} +  t^{2} - 2  t + 1}
& - 
& -
\\
\hline
(3,4)
&\displaystyle \frac{ t^{5} -  t^{4} +  t^{2}}{ t^{5} - 2  t^{4} +  t^{3} +  t^{2} - 2  t + 1}
& -
& -
\\
\hline
\end{array}
\]
}
The sum of the above functions weighted with the dimensions of the corresponding representations gives the equivariant Hilbert function
\begin{eqnarray*}
H_{R,G}=\frac{ t^3 + 8t^2 + 8 t + 1}{ t^2 -2 t + 1}
\end{eqnarray*}
We can verify that our computations are correct, since $H(R,G)$ has Taylor expansion
\begin{multline*}
1 + 10  t + 27  t^{2} + 45  t^{3} + 63  t^{4} + 81  t^{5} + 99  t^{6} + 117  t^{7} + 135  t^{8} + 153  t^{9} + 171  t^{10} + 189  t^{11} + 207  t^{12} \\
+ 225  t^{13} + 243  t^{14} + 261  t^{15} + 279  t^{16} + 297  t^{17} + 315  t^{18} + 333  t^{19 }+ O( t^{20}),
\end{multline*}
and we note that the coefficient of $t$ is $g=10$, whereas for $m\geq 2$ the coefficient of $t^m$ is $9(2m-1)$ as expected.
\end{example}

 \def\cprime{$'$}


\begin{thebibliography}{10}

\bibitem{MR2590895}
Michel Brou\'{e}.
\newblock {\em Introduction to complex reflection groups and their braid
  groups}, volume 1988 of {\em Lecture Notes in Mathematics}.
\newblock Springer-Verlag, Berlin, 2010.

\bibitem{charalambous2019relative}
Hara {Charalambous}, Kostas {Karagiannis}, and Aristides {Kontogeorgis}.
\newblock {The Relative Canonical Ideal of the Artin-Schreier-Kummer-Witt
  family of curves}.
\newblock {\em arXiv e-prints}, page arXiv:1905.05545, May 2019.

\bibitem{Chevalley1934-eb}
C~Chevalley, A~Weil, and E~Hecke.
\newblock {\"{U}}ber das verhalten der integrale 1. gattung bei automorphismen
  des funktionenk{\"{o}}rpers.
\newblock {\em Abh. Math. Semin. Univ. Hambg.}, 1934.

\bibitem{MR1274097}
Ted Chinburg.
\newblock Galois structure of de {R}ham cohomology of tame covers of schemes.
\newblock {\em Ann. of Math. (2)}, 139(2):443--490, 1994.

\bibitem{MR589254}
G.~Ellingsrud and K.~L\o~nsted.
\newblock An equivariant {L}efschetz formula for finite reductive groups.
\newblock {\em Math. Ann.}, 251(3):253--261, 1980.

\bibitem{GJKL}
Dimos Goundaroulis, Jes\'{u}s Juyumaya, Aristides Kontogeorgis, and Sofia
  Lambropoulou.
\newblock Framization of the {T}emperley-{L}ieb algebra.
\newblock {\em Math. Res. Lett.}, 24(2):299--345, 2017.

\bibitem{MR2525558}
Frank Himstedt and Peter Symonds.
\newblock Equivariant {H}ilbert series.
\newblock {\em Algebra Number Theory}, 3(4):423--443, 2009.

\bibitem{Kapranov1995-xx}
MM~Kapranov and A~Smirnov.
\newblock Cohomology determinants and reciprocity laws.
\newblock 1995.

\bibitem{KaranProc}
Sotiris Karanikolopoulos and Aristides Kontogeorgis.
\newblock Integral representations of cyclic groups acting on relative
  holomorphic differentials of deformations of curves with automorphisms.
\newblock {\em Proc. Amer. Math. Soc.}, 142(7):2369--2383, 2014.

\bibitem{Koeck:04}
Bernhard K{\"o}ck.
\newblock Galois structure of {Z}ariski cohomology for weakly ramified covers
  of curves.
\newblock {\em Amer. J. Math.}, 126(5):1085--1107, 2004.

\bibitem{190910282}
Aristides Kontogeorgis, Alexios Terezakis, and Ioannis Tsouknidas.
\newblock Automorphisms and the canonical ideal.
\newblock {\em Mediterranean Journal of Mathematics (accepted)}, page~16, 2021.

\bibitem{Leopoldt:96}
Heinrich-Wolfgang Leopoldt.
\newblock \"{U}ber die {A}utomorphismengruppe des {F}ermatk\"orpers.
\newblock {\em J. Number Theory}, 56(2):256--282, 1996.

\bibitem{Nak:86}
Sh{\=o}ichi Nakajima.
\newblock Action of an automorphism of order {$p$} on cohomology groups of an
  algebraic curve.
\newblock {\em J. Pure Appl. Algebra}, 42(1):85--94, 1986.

\bibitem{MR1753111}
Georgios Pappas.
\newblock Galois module structure and the {$\gamma$}-filtration.
\newblock {\em Compositio Math.}, 121(1):79--104, 2000.

\bibitem{Saint-Donat73}
B.~Saint-Donat.
\newblock On {P}etri's analysis of the linear system of quadrics through a
  canonical curve.
\newblock {\em Math. Ann.}, 206:157--175, 1973.

\bibitem{SerreLinear}
Jean-Pierre Serre.
\newblock {\em Linear representations of finite groups}.
\newblock Springer-Verlag, New York, 1977.
\newblock Translated from the second French edition by Leonard L. Scott,
  Graduate Texts in Mathematics, Vol. 42.

\bibitem{MR526968}
Richard~P. Stanley.
\newblock Invariants of finite groups and their applications to combinatorics.
\newblock {\em Bull. Amer. Math. Soc. (N.S.)}, 1(3):475--511, 1979.

\bibitem{sage8.9}
W.\thinspace{}A. Stein et~al.
\newblock {\em {S}age {M}athematics {S}oftware ({V}ersion 8.9)}.
\newblock The Sage Development Team, 2019.
\newblock {\tt http://www.sagemath.org}.

\bibitem{MR2304322}
Peter Symonds.
\newblock Structure theorems over polynomial rings.
\newblock {\em Adv. Math.}, 208(1):408--421, 2007.

\bibitem{MR608528}
M.~J. Taylor.
\newblock On {F}r\"{o}hlich's conjecture for rings of integers of tame
  extensions.
\newblock {\em Invent. Math.}, 63(1):41--79, 1981.

\bibitem{Tze:95}
Pavlos Tzermias.
\newblock The group of automorphisms of the {F}ermat curve.
\newblock {\em J. Number Theory}, 53(1):173--178, 1995.

\end{thebibliography}
\end{document}